\documentclass[11pt]{article}
\usepackage[latin1]{inputenc}
\usepackage[T1]{fontenc}
\usepackage{amsmath,amssymb,amstext}
\usepackage{theorem}
\usepackage{multicol}
\usepackage[english]{babel}
\usepackage{enumerate}
\usepackage{eufrak}
\usepackage{mathrsfs}

\newtheorem{theorem}{Theorem}

\newtheorem{corollary}[theorem]{Corollary}

\newtheorem{lemma}[theorem]{Lemma}

\newtheorem{proposition}[theorem]{Proposition}
\newtheorem{remark}[theorem]{Remark}

\newenvironment{proof}[1][Proof]{\textbf{#1.} }{\ \rule{0.5em}{0.5em}}

\renewcommand{\geq}{\geqslant}

\def\1{{\mathbf{1}}}

\def\1{{\mathbf{1}}}
\def\0.5{{\frac{1}{2}}}

\renewcommand{\thefootnote}{\fnsymbol{footnote}}
\begin{document}

\renewcommand{\thefootnote}{\arabic{footnote}}

\begin{center}
{\Large \textbf{Optimal rates for parameter estimation of stationary
Gaussian processes }} \\[0pt]
~\\[0pt]
Khalifa Es-Sebaiy\footnote{%
National School of Applied Sciences - Marrakesh, Cadi Ayyad University,
Marrakesh, Morocco. Email: \texttt{k.essebaiy@uca.ma}} and Frederi G. Viens
\footnote{%
Dept. Statistics and Dept. Mathematics, Purdue University, 150 N. University
St., West Lafayette, IN 47907-2067, USA. E-mail: \texttt{viens@purdue.edu}}\\%
[0pt]
\textit{Cadi Ayyad University and Purdue University}\\[0pt]
~\\[0pt]
\end{center}

{\ \noindent \textbf{Abstract:} We study rates of convergence in central
limit theorems for partial sum of functionals of general stationary and
non-stationary Gaussian sequences, using optimal tools from analysis on
Wiener space. We apply our result to study drift parameter estimation
problems for some stochastic differential equations driven by fractional
Brownian motion with fixed-time-step observations.\vspace*{0.1in}}



\section{Introduction}

While statistical inference for Itô-type diffusions has a long history,
statistical estimation for equations driven by fractional Brownian motion
(fBm) is much more recent, partly because the development of stochastic
calculus with respect to the fBm, which provides tools to study such models,
is itself a recent and ongoing endeavor, and partly because these tools can
themselves be unwieldy in comparison with the convenience and power of
martingale methods and the Markov property which accompany Itô models. Our
purpose in this article is to show how the analysis on Wiener space,
particularly via tools recently developed to study the convergence-in-law
properties in Wiener chaos, can be brought to bear on parameter estimation
questions for fBm-driven models, and more generally for arbitrary stationary
Gaussian models.

\subsection{Context and general ideas}

There are several approaches to estimating drift parameters in fBm-driven
models, which have been developed over the course of the past 10 or 15
years. The approaches we mention below are related to the methods in this
article.

\begin{itemize}
\item The MLE approach in \cite{KL}, \cite{TV}. In general the techniques
used to construct maximum likelihood estimators (MLE) for drift parameters
are based on Girsanov transforms for fBm and depend on the properties of the
deterministic fractional operators (determined by the Hurst parameter)
related to the fBm. In general, the MLE is not easily computable. In
particular, it relies on being able to compute stochastic integrals with
respect to fBm. This is difficult or hopeless for most models since
approximating pathwise integrals w.r.t. fBm, when they exist, is
challenging, while Skorohod-type integrals cannot be computed based on the
data except in special cases. The work in \cite{TV} is the only one in which
a strongly consistent discretization of the MLE was based on long-horizon
asymptotics without also requiring an in-fill (small time step) condition,
though it did not establish any asymptotic distribution.

\item A least-squares (LS) approach was proposed in \cite{HN}. The study of
the asymptotic properties of the estimator is based on certain criteria
formulated in terms of the Malliavin calculus (see \cite{NP-book}). It
should be noted that in \cite{HN}, the full LS estimator relies on an
unobservable Skorohod integral, and the authors proposed a modified version
of this estimator which can be computed based on in-fill asymptotics;
however, this modified estimator bears no immediate relation to an LS one
(see \cite{EEV} for examples of what constitutes a discretization of an LS
estimator for fBm models, and for a comparison with MLE methods, which
coincide with LS methods if and only if $H=1/2$). In the ergodic case, the
statistical inference for several fractional Ornstein-Uhlenbeck (fOU) models
via LS methods was recently developed in the papers \cite{HN}, \cite{AM},
\cite{AV}, \cite{EEV}, \cite{HS}, \cite{BI}, \cite{NT}. The case of
non-ergodic fOU process of the first kind and of the second kind can be
found in \cite{BEO}, \cite{EEO} and \cite{EET} respectively.
\end{itemize}

We bring new techniques to statistical inference for stochastic differential
equations (SDEs) related to stationary Gaussian processes. Some of these
ideas can be summarized as follows:

\begin{itemize}
\item Since the theory of inference for these fBm-driven SDEs is still near
its inception, and most authors are concerned with linear problems, whose
solutions are Gaussian, this Gaussian property should be exploited to its
fullest extent, given the best tools currently available.

\begin{itemize}
\item Therefore we choose to consider polynomial variations of these
processes, which then necessarily live in Wiener chaos, whose properties are
now well understood thanks to new Malliavin-calculus advances which were
initiated by Nourdin and Peccati in 2008; in particular, we rely a general
observation and their so-called optimal 4th moment theorem, in \cite{NP2013}.

\item As a consequence, we are able to compute upper bounds in the total
variation (TV) norm for the rate of normal convergence of our estimator. In
particular, for the quadratic case, we prove a Berry-Esséen theorem (speed
on the order of $1/\sqrt{n}$) for this TV norm which we show is sharp in
some cases by finding a lower bound with the same speed. No authors as far
as we know have ever provided such quantitative estimates of the speed of
asymptotic normality for any drift estimators for any fBm-driven model, let
alone shown that they are sharp.
\end{itemize}

\item Rather than starting from the continuous-time setting of SDEs, and
then attempt to discretize resulting LS estimators, as was done in many of
the aforementioned works including our own \cite{EEV}, we work from
discretely observed data from the continuous-time SDEs, and design
estimators based on such Gaussian sequences. In fact, we show that one can
develop estimators valid for any Gaussian sequence, with suitable conditions
on the sequence's auto-correlation function, and then apply them to
fBm-driven SDEs of interest. In this way, we are able to provide estimators
for many other models, while the models studied in \cite{HS}, \cite{AM},
\cite{AV}, \cite{EEV} become particular cases in our approach.

\item Since our method relies on conditions which need only be checked
intrinsically on the auto-correlation function, it can apply equally well to
in-fill situations and increasing-horizon situations.

\begin{itemize}
\item It turns out that, as an artefact of trying to discretize estimators
based on continuous paths, prior works were never able to avoid an in-fill
assumption on the data (and sometimes even required both in-fill and
increasing-horizon assumptions). In this paper, we illustrate our methods by
showing that in-fill assumptions are never needed for the examples we cover.

\item Essentially, as explained in more detail further below, if a Gaussian
stochastic process has a memory correlation length which is bounded above by
that of a fBm with Hurst parameter $H<3/4$, then our polynomial variations
estimator based on discrete data (fixed time step) is asymptotically normal
as the number of observations $n$ increases, with a TV speed as good as $1/%
\sqrt{n}$, as mentioned above.
\end{itemize}

\item Finally, we provide a systematic study of how to go from stationary
observations, to observations coming from a Gaussian process which may not
be stationary, by implementing a fully quantitative strategy to control the
contribution of the non-stationarity to the TV convergence speeds. In the
examples we cover, which are those of recent interest in the literature, the
non-stationarity term vanishes exponentially fast, which is more than enough
for our generic condition to hold, but slower power convergences would yield
the same results, for summable powers.
\end{itemize}

\subsection{Summary of results}

We summarize our paper's contents briefly in this section, including some
heuristics for easier reading. Consider a centered stationary Gaussian
process $Z=\left( Z_{k}\right) _{k\in \mathbb{Z}}$ with covariance
(auto-correlation function)
\begin{equation*}
r_{Z}(k):=\mathbf{E}\left[ Z_{0}Z_{k}\right] \mbox{ for every }k\in \mathbb{Z%
}\mbox{ such that }r_{Z}(0)>0.
\end{equation*}%
Fix a polynomial function $f_{q}$ of degree $q$ where $q$ is an even
integer. To estimate the parameter $\lambda _{f_{q}}(Z):=\mathbf{E}\left[
f_{q}(Z_{0})\right] $, we use the \textquotedblleft polynomial
variation\textquotedblright\ estimator
\begin{equation*}
Q_{f_{q},n}(Z):=\frac{1}{n}\sum_{i=0}^{n-1}f_{q}(Z_{i}),
\end{equation*}%
which can be considered as a scalar version of a generalized method of
moments. Some of the general results we prove are the following.

\begin{itemize}
\item $Q_{f_{q},n}(Z)$ is strongly consistent under a very weak decay
condition on $r_{Z}$ (Theorem \ref{consistency for Z}), without requiring
ergodicity.

\item To avoid situations where the memory of $Z$ is so long that $%
Q_{f_{q},n}(Z)$'s asymptotics are non-normal, we introduce the following
assumption (condition (\ref{BM})), which is a special case of the condition
used by Breuer and Major\ in 1983 to establish normality of Hermite
variations (see \cite[Chapter 7]{NP-book}):%
\begin{equation*}
u_{f_{2}}\left( Z\right) :=2\sum_{j\in \mathbb{Z}}r_{Z}(j)^{2}<\infty ,
\end{equation*}%
with a similar definition for $\lambda _{f_{q}}(Z)$, the limit of $Var\left(
Q_{f_{q},n}(Z)\right) $, which is then also finite.

\begin{itemize}
\item Under this condition, as an extension of our main TV convergence rate
Theorem \ref{CLT}, we prove the following for the normalized (Corollary \ref%
{CLTCor}):%
\begin{eqnarray}
&&d_{TV}\left( \sqrt{n}\left[ Q_{f_{q},n}(Z)-\lambda _{f_{q}}(Z)\right] ,%
\mathcal{N}\left( 0,u_{f_{q}}(Z)\right) \right)  \notag \\
&\leqslant &C_{q}(Z)\left( \sqrt[4]{\kappa _{4}(U_{f_{2},n}(Z))}+\sqrt{%
\kappa _{4}(U_{f_{2},n}(Z))}\right) +2\left\vert 1-\frac{Var\left(
Q_{f_{q},n}(Z)\right) }{u_{f_{q}}(Z)}\right\vert  \label{previewCLTcor}
\end{eqnarray}%
where $U_{f_{2},n}(Z):=Q_{f_{2},n}(Z)-u_{f_{2}}\left( Z\right) $, which
comes from the case $q=2$, controls all cases of $q$ nonetheless. Thus the
total variation distance between the renormalized estimator and the normal
law with asymptotic variance $u_{f_{q}}(Z)$ is bounded by the $4$th root of
the $4$th cumulant in the quadratic case, and the relative distance of the
estimator's variance to its limit.

\item If further normalizing by $Var\left( Q_{f_{q},n}(Z)\right) $, the last
term above vanishes, though the unnormalized expression is the only one
which can be computed in practice, since $Var\left( Q_{f_{q},n}(Z)\right) $
depends on the parameter $\lambda $. Thus the speed of convergence of $%
Var\left( Q_{f_{q},n}(Z)\right) $ is of major practical importance.

\item There are explicit expressions for $\kappa _{4}(U_{f_{2},n}(Z))$ and $%
Var\left( Q_{f_{q},n}(Z)\right) $ which can be expressed using $r_{Z}$, as
explained in Section \ref{EXGEN}. Consequently, the above upper bound can be
computed explicitly for many cases of $r_{Z}$. For instance (Corollary \ref%
{CLTCor2}) , if $Z$ has a memory which is bounded above by that of
fractional Gaussian noise with parameter $H<5/8$, the above result yields%
\begin{equation*}
d_{TV}\left( \sqrt{n}\left[ Q_{f_{q},n}(Z)-\lambda _{f_{q}}(Z)\right] ,%
\mathcal{N}\left( 0,u_{f_{q}}(Z)\right) \right) \leqslant 1/\sqrt[4]{n}.
\end{equation*}
\end{itemize}

\item The case where $Q_{f_{q},n}(Z)$'s asymptotics are non-normal can be
treated using other, less optimal, tools. For the sake of conciseness, we do
not provide detailed arguments in this paper, instead stating results
without proof in Remarks \ref{NCLT} and \ref{NNCLT}.

\item In practice, it is common to encounter situations where observations
are not stationary, for instance because their initial value is a point mass
rather than the stationary law of a stochastic system. Thus, assuming that
the observations come from $X_{k}=Z_{k}+Y_{k}$ where $Z$ is as above and $Y$
is the deviation from a stationary process, we prove a convergence theorem
under a generic assumption on $Y$ which is easily verified in practice. If
there exist $p_{0}\in \mathbb{N}$ and a constant $\gamma >0$ such that for
every $p\geq p_{0}$ and for all $n\in \mathbb{N}$,
\begin{equation*}
\left\Vert Q_{f_{q},n}(Z+Y)-Q_{f_{q},n}(Z)\right\Vert _{L^{p}(\Omega )}=%
\mathcal{O}\left( n^{-\gamma }\right)
\end{equation*}%
then for the Wasserstein distance (see Theorem \ref{CLT for Z+Y} for
details)
\begin{eqnarray*}
&&d_{W}\left( Q_{f_{q},n}(X)-\lambda _{f_{q}}(Z),\mathcal{N}\left(
0,u_{f_{q}}(Z)\right) \right) \\
&\leqslant &Cn^{\frac{1}{2}-\gamma }+C\sqrt[4]{\kappa _{4}(U_{f_{2},n}(Z))}%
+C\left\vert 1-\frac{Var\left( Q_{f_{q},n}(Z)\right) }{u_{f_{q}}(Z)}%
\right\vert .
\end{eqnarray*}

\item When $q=2$, up to a constant, the estimator $Q_{f_{2},n}(Z)$ is in the
second chaos. In this case, sharper results are established.

\begin{itemize}
\item For instance, assuming $u_{f_{2}}\left( Z\right) <\infty $ and the
following two conditions (see Theorem \ref{CLT for quadratic var} for
details):%
\begin{eqnarray*}
\left\Vert Q_{f_{2},n}(Z+Y)-Q_{f_{2},n}(Z)\right\Vert _{L^{1}(\Omega )}
&\leqslant &\mathcal{O}\left( \frac{1}{n}\right) , \\
\left\vert u_{f_{2}}\left( Z\right) -E\left[ U_{f_{2},n}^{2}(Z)\right]
\right\vert &\leqslant &\mathcal{O}\left( \frac{1}{\sqrt{n}}\right) ,
\end{eqnarray*}%
then for the Wasserstein distance, assuming constants in the above
assumptions are not too large,%
\begin{equation*}
\frac{c_{1}}{\sqrt{n}}\leqslant d_{W}\left( \sqrt{n}\left[
Q_{f_{q},n}(X)-\lambda _{f_{q}}(Z)\right] ,\mathcal{N}\left(
0,u_{f_{q}}(Z)\right) \right) \leqslant \frac{C_{1}}{\sqrt{n}}.
\end{equation*}%
In this sense, we have established conditions under which our variation for
a non-stationary and highly correlated sequence satisfies a quantitative
Berry-Esséen-type theorem with optimal rate (recall that the classical
Berry-Esséen theorem is for an i.i.d. sequence and is stated for the
Kolmogorov distance, which is bounded above strictly by our Wasserstein
distance).

\item It is remarkable that this results holds for all Gaussian sequences
with autocorrelation bounded above by that of fBm with $H<2/3$ (whereas the
best results for $q>2$ show that one needs the stronger condition $H<5/8$).
\end{itemize}

\item Before moving to specific examples, we establish two improvements: a
strategy for converting the above results into estimators for parameters
which are buried in a functional form for $\lambda _{f_{q}}(Z)$, and a
method for improving rates of convergence by taking finite differences.

\begin{itemize}
\item If we are interested in a parameter $\theta $ which is related to $%
\lambda $ via $\lambda _{f_{q}}(Z)=g^{-1}\left( \theta \right) $ where $g$
is a diffeomorphism, so that $\check{\theta}_{n}:=g\left(
Q_{f_{q},n}(Z)\right) $ is a strongly consistent estimator of $\theta $,
under the condition that $g^{\prime \prime }\left( Q_{f_{q},n}(Z)\right) $
has moments of sufficiently large order, then%
\begin{equation*}
d_{W}\left( \sqrt{n}\left( \check{\theta}_{n}-\theta \right) ,\mathcal{N}%
(0,g^{\prime }(\lambda _{f_{q}}(Z))^{2}Var\left( Q_{f_{q},n}(Z)\right)
)\right)
\end{equation*}%
converges to $0$ at the same speed as in (\ref{previewCLTcor}). This is in
Theorem \ref{Berryesseentheta}.

\item Finally, by considering $X_{k}^{\left( 1\right) }=X_{k}-X_{k-1}$, it
is well known that memory length is decreased by 2 power units in
autocorrelation for all long-memory sequences with power decay; thus $%
Z_{k}^{\left( 1\right) }$ becomes sufficiently short memory to allow us to
apply the best convergence results above; in particular conditions such as
\textquotedblleft $H<5/8$\textquotedblright\ are automatically satisfied as
soon as one takes a first-order finite difference. This is explained in
Section \ref{IRC}. For instance in Theorem \ref{OptimalImprove}, we find
\begin{equation*}
\frac{c}{\sqrt{n}}\leqslant d_{W}\left( Q_{f_{2},n}(Z^{\left( 1\right)
})-\lambda _{f_{2},n}(Z^{\left( 1\right) }),\mathcal{N}\left(
0,u_{f_{2}}(Z^{\left( 1\right) })\right) \right) \leqslant \frac{C}{\sqrt{n}}%
.
\end{equation*}%
Arguably, as long as one can compute $\lambda _{f_{2},n}(Z^{\left( 1\right)
})$ and relate it to a parameter of interest, this improvement allows one to
take advantage of the best rate of convergence, that of Berry-Esséen order.
A study of one example of what it means to extract a parameter from $\lambda
_{f_{2},n}(Z^{\left( 1\right) })$ is given in Section \ref{HigherFOU}, for
the fractional Ornstein-Uhlenbeck process.
\end{itemize}
\end{itemize}

This bring us to the last sections in which we apply the above results to
specific cases.

\begin{itemize}
\item \emph{Application to fractional Ornstein Uhlenbeck models}. An
Ornstein-Uhlenbeck process $X=\left\{ X_{t},t\geq 0\right\} $ is the
solution of the linear stochastic differential equation
\begin{equation}
X_{0}=0;\quad dX_{t}=-\theta X_{t}dt+dG_{t},\quad t\geq 0,  \label{OUintro}
\end{equation}%
where $G$ is a Gaussian process and $\theta >0$ is an unknown parameter. The
problem here is to estimate the parameter $\theta $ based on discrete
equidistant observations (fixed time step, horizon tending to $+\infty $),
and provide precise CLTs, which can be useful for hypothesis testing in
practice via parametric inference.

\begin{itemize}
\item \emph{Fractional Ornstein-Uhlenbeck model }(Section \ref{AOUP}): the
process $G$ in (\ref{OUintro}) is a fractional Brownian motion with Hurst
parameter $H\in (0,1)$. By assuming that $H>\frac{1}{2}$, \cite{HN} studied
the least-squares estimator (LSE) $\widehat{\theta }_{t}=\left(
\int_{0}^{t}X_{s}\delta X_{s}\right) /\left( \int_{0}^{t}X_{s}^{2}ds\right) $
of $\theta $ when the process $X$ is continuously observed. In this paper,
using our approach we construct a class of explicit estimators of $\theta $
when the process $X$ is discretely observed. We study the asymptotic
behavior of these estimators for any $H\in (0,1)$ with Berry-Esséen-type
theorems. We prove the consistency and prove that our estimators are
asymptotically normal when $H\in (0,\frac{3}{4}]$. In the particular case
when $q=2$ and $f_{2}(x)=x^{2}$, \cite{HS} proved strong consistency of this
discrete estimator and gave a Berry-Esséen-type result when $H\in (\frac{1}{2%
},\frac{3}{4})$ but the proofs in \cite{HS} rely on a possibly flawed
technique, since the passage from line -7 to -6 on page 434 is true if $H>%
\frac{3}{4}$, while one expects normal asymptotics only for the case $%
H\leqslant \frac{3}{4}$. Our work resolves the issue of $\theta $ estimation
via least squares and their higher-order generalizations, by appealing to
our new tools, avoiding the arguments in \cite{HS}. We present a number of
results in this section, including Berry-Esséen estimates, optimal lower
bounds thereof, an implementation the inversion of the quadratic variation
estimator to access $\theta $ directly, and details of how to increase the
rate of convergence via finite-differences.

\item \emph{Ornstein-Uhlenbeck process driven by fractional
Ornstein-Uhlenbeck process} (Section \ref{OUFOUsection}):{\small \ }here $X\
$is again given by (\ref{OUintro}) with drift parameter $\theta $, and $G$
is itself a fractional Ornstein-Uhlenbeck process with Hurst index $H\in
(0,1)$ and drift parameter $\rho >0$. Here $\theta $ and $\rho $ are
considered as unknown parameters (with $\theta \neq \rho $), and we assume
that only $X$ is observed. This is the long-memory analogue of a
continuous-time latent Markovian framework, in other words a partial
observation, or partial information, question. This question was considered
in \cite{EEV} when $H\in (\frac{1}{2},\frac{3}{4})$; therein, an estimator
of $(\theta ,\rho )$ was provided in the continuous and discrete cases, but
relatively strong in-fill assumptions were needed, though the estimators
still needed an increasing-horizon setting. In the present paper, we extend
the result to $H\in (0,1)$ and we propose a class of estimators with
Berry-Esséen behavior, which dispenses with any in-fill assumption. A full
set of results such as in Section \ref{OUFOUsection} could also be derived,
including optimal Berry-Esséen rates for the quadratic case in this
two-dimensional setting; for the sake of conciseness, we omit stating all
these improvements.

\item \emph{Fractional Ornstein-Uhlenbeck process of the second kind}
(Section \ref{FOUSKsection}):\emph{\ }again with $X\ $as in (\ref{OUintro}),
this process arises when $G$ has the form $G_{t}=%
\int_{0}^{t}e^{-s}dB_{a_{s}} $ with $a_{s}=He^{\frac{s}{H}}$ and $B=\left\{
B_{t},t\geq 0\right\} $ is a fractional Brownian motion with Hurst parameter
$H\in (\frac{1}{2},1)$, and where $\theta >0$ is a unknown real parameter;
notationally, to be consistent with previous work on this topic, we use the
letter $\alpha $ instead of $\theta $. The continuous and discrete cases,
when $q=2$ and $f_{2}(x)=x^{2}$, are studied in \cite{AM} and \cite{AV},
though no speeds of convergence are provided, and in-fill assumptions are
needed. In Section \ref{FOUSKsection}, we propose a class of estimators and
we provide Berry-Esséen-type theorems of them with no in-fill assumption.
The covariance structure of the process is such that our methods easily
provide an optimal convergence rate for the estimator after inversion of the
quadratic variation.\bigskip
\end{itemize}
\end{itemize}

In conclusion, our methodology is developed for essentially any stationary
Gaussian sequence, we can handle some non-stationarity under a weak
assumption on the speed of relaxation to a stationary law, we provide
Berry-Esséen rates for the normal asymptotics of our polynomial variation
estimators, particularly in the quadratic case where the rates are often
optimal, and we analyze some of the issues that can arise when inverting a
polynomial variation to access a specific parameter. This is all achieved by
relying on the sharpest estimates known to date, in the framework of Nourdin
and Peccati, in Wiener chaos for total variation and Wasserstein convergence
in law. Applications to drift estimation for long-memory models of current
interest are provided.\bigskip

Our article is structured as follows. Section \ref{Wiener} provides some
basic elements of analysis on Wiener space which are helpful for some of the
arguments we use. Section \ref{PARAMESTIM} provides the general theory of
polynomial variation for general Gaussian sequences, covering the stationary
case (Section \ref{AD}, with examples in Section \ref{EX}), non-stationary
cases (Section \ref{NONSTAT}), which include optimality in the quadratic
case even under non-stationarity (Section \ref{QC}) and a strategy of how to
access a specific parameter other than the polynomial's variance (Section %
\ref{Towards}). Section \ref{IRC} explains under what circumstances one can
increase the rate of convergence to an optimal level by finite-differencing.
Finally, three sets of examples based on fractional Ornstein-Uhlenbeck
constructions are given in Sections \ref{AOUP} and \ref{Multi}. Some of the
technical results used in various proofs, including the proof of the basic
Berry-Esséen theorem in the stationary case, are in the Appendix (Section %
\ref{Appendix}).

\section{Elements of analysis on Wiener space\label{Wiener}}

Here we summarize a few essential facts from the analysis on Wiener space
and the Malliavin calculus. Though these facts and notation are essential
underpinnings of the tools and results of this paper, most of our results
and arguments can be understood without knowledge of the elements in this
section. The interested reader can find more details in \cite[Chapter 1]%
{nualart-book} and \cite[Chapter 2]{NP-book}.

Let $\left( \Omega ,\mathcal{F},\mathbf{P}\right) $ be a standard Wiener
space, its standard Wiener process $W$, where for a deterministic function $%
h\in L^{2}\left( \mathbf{R}_{+}\right) =:{{\mathcal{H}}}$, the Wiener
integral $\int_{\mathbf{R}_{+}}h\left( s\right) dW\left( s\right) $ is also
denoted by $W\left( h\right) $. The inner product $\int_{\mathbf{R}%
_{+}}f\left( s\right) g\left( s\right) ds$ will be denoted by $\left\langle
f,g\right\rangle _{{\mathcal{H}}}$. For every $q\geq 1$, let ${\mathcal{H}}%
_{q}$ be the $q$th Wiener chaos of $W$, that is, the closed linear subspace
of $L^{2}(\Omega )$ generated by the random variables $\{H_{q}(W(h)),h\in {{%
\mathcal{H}}},\Vert h\Vert _{{\mathcal{H}}}=1\}$ where $H_{q}$ is the $q$th
Hermite polynomial. The mapping ${I_{q}(h^{\otimes q}):}=q!H_{q}(W(h))$
provides a linear isometry between the symmetric tensor product ${\mathcal{H}%
}^{\odot q}$ (equipped with the modified norm $\Vert .\Vert _{{\mathcal{H}}%
^{\odot q}}=\frac{1}{\sqrt{q!}}\Vert .\Vert _{{\mathcal{H}}^{\otimes q}}$)
and ${\mathcal{H}}_{q}$. It also turns out that ${I_{q}(h^{\otimes q})}$ is
the multiple Wiener integral of ${h^{\otimes q}}$ w.r.t. $W$. For every $%
f,g\in {{\mathcal{H}}}^{\odot q}$ the following product formula holds
\begin{equation*}
E\left( I_{q}(f)I_{q}(g)\right) =q!\langle f,g\rangle _{{\mathcal{H}}%
^{\otimes q}}.
\end{equation*}%
For $h\in {\mathcal{H}}^{\otimes q}$, the multiple Wiener integrals $%
I_{q}(h) $, which exhaust the set ${\mathcal{H}}_{q}$, satisfy a
hypercontractivity property (equivalence in ${\mathcal{H}}_{q}$ of all $L^{p}
$ norms for all $p\geq 2$), which implies that for any $F\in \oplus
_{l=1}^{q}{\mathcal{H}}_{l}$, we have
\begin{equation}
\left( E\big[|F|^{p}\big]\right) ^{1/p}\leqslant c_{p,q}\left( E\big[|F|^{2}%
\big]\right) ^{1/2}\ \mbox{ for any }p\geq 2.  \label{hypercontractivity}
\end{equation}

Though we will not insist on their use in the main body of the paper,
leaving associated technicalities to the proof of one of our main theorems
in the appendix, the Malliavin derivative operator $D$ on $L^{2}\left(
\Omega \right) $ plays a fundamental role in evaluating distances between
random variables therein. For any function $\Phi \in C^{1}\left( \mathbf{R}%
\right) $ with bounded derivative, and any $h\in {\mathcal{H}}$, we define
the Malliavin derivative of the random variable $X:=\Phi \left( W\left(
h\right) \right) $ to be consistent with the following chain rule:%
\begin{equation*}
DX:X\mapsto D_{r}X:=\Phi ^{\prime }\left( W\left( h\right) \right) h\left(
r\right) \in L^{2}\left( \Omega \times \mathbf{R}_{+}\right) .
\end{equation*}%
A similar chain rule holds for multivariate $\Phi $. One then extends $D$ to
the so-called Gross-Sobolev subset $\mathbf{D}^{1,2}\varsubsetneqq
L^{2}\left( \Omega \right) $ by closing $D$ inside $L^{2}\left( \Omega
\right) $ under the norm defined by
\begin{equation*}
\left\Vert X\right\Vert _{1,2}^{2}=\mathbf{E}\left[ X^{2}\right] +\mathbf{E}%
\left[ \int_{\mathbf{R}_{+}}\left\vert D_{r}X\right\vert ^{2}dr\right] .
\end{equation*}

Now recall that, if $X,Y$ are two real-valued random variables, then the
total variation distance between the law of $X$ and the law of $Y$ is given
by
\begin{equation*}
d_{TV}\left( X,Y\right) =\sup_{A\in \mathcal{B}({\mathbb{R}})}\left\vert P%
\left[ X\in A\right] -P\left[ Y\in A\right] \right\vert .
\end{equation*}%
If $X,Y$ are two real-valued integrable random variables, then the
Wasserstein distance between the law of $X$ and the law of $Y$ is given by
\begin{equation*}
d_{W}\left( X,Y\right) =\sup_{f\in Lip(1)}\left\vert Ef(X)-Ef(Y)\right\vert
\end{equation*}%
where $Lip(1)$ indicates the collection of all Lipschitz functions with
Lipschitz constant $\leqslant 1$. Let $N$ denote the standard normal law.
All Wiener chaos random variable are in the domain $\mathbf{D}^{1,2}$ of $D$%
, and are orthogonal in $L^{2}\left( \Omega \right) $. The so-called Wiener
chaos expansion is the fact that any $X\in \mathbf{D}^{1,2}$ can be written
as $X=\mathbf{E}X+\sum_{q=0}^{\infty }X_{q}$ where $X_{q}\in {\mathcal{H}}%
_{q}$. We define a linear operator $L$ which is diagonalizable under the ${%
\mathcal{H}}_{q}$'s by saying that ${\mathcal{H}}_{q}$ is the eigenspace of $%
L$ with eigenvalue $-q$, i.e. for any $X\in {\mathcal{H}}_{q}$, $LX=-qX$.
The kernel of $L$ is the constants. The operator $-L^{-1}$ is the negative
pseudo-inverse of $L$, so that for any $X\in {\mathcal{H}}_{q}$, $%
-L^{-1}X=q^{-1}X$. Since the variables we will be dealing with in this
article are finite sums of elements of ${\mathcal{H}}_{q}$, the operator $%
-L^{-1}$ is easy to manipulate thereon.

Two key estimates linking total variation distance and the Malliavin
calculus are the following.

\begin{itemize}
\item Let $X\in \mathbf{D}^{1,2}$ with $\mathbf{E}\left[ X\right] =0$. Then
(see \cite[Proposition 2.4]{NP2013}),%
\begin{equation*}
d_{TV}\left( X,N\right) \leqslant 2E\left\vert 1-\left\langle
DX,-DL^{-1}X\right\rangle _{\mathcal{H}}\right\vert .
\end{equation*}

\item Let a sequence $X:X_{n}\in {\mathcal{H}}_{q}$, such that $\mathbf{E}%
X_{n}=0$ and $Var\left[ X_{n}\right] =1$ , and assume $X_{n}$ converges to a
normal law in distribution, which is equivalent to $\lim_{n}\mathbf{E}\left[
X_{n}^{4}\right] =3$ (this equivalence, proved originally in \cite{NP2005},
is known as the \emph{fourth moment theorem}). Then we have the following
optimal estimate for $d_{TV}\left( X,N\right) $, known as the optimal 4th
moment theorem, proved in \cite{NP2013}: there exist two constant $c,C>0$
depending only on the sequence $X$ but not on $n$, such that
\begin{equation*}
c\max \left\{ \mathbf{E}\left[ X_{n}^{4}\right] -3,\left\vert \mathbf{E}%
\left[ X_{n}^{3}\right] \right\vert \right\} \leqslant d_{TV}\left(
X,N\right) \leqslant C\max \left\{ \mathbf{E}\left[ X_{n}^{4}\right]
-3,\left\vert \mathbf{E}\left[ X_{n}^{3}\right] \right\vert \right\} .
\end{equation*}
\end{itemize}

Given the importance of the centered 4th moment, also known as a 4th
cumulant, of a standardized random variable, we will use the following
special notation:%
\begin{equation*}
\kappa _{4}\left( X\right) :=\mathbf{E}\left[ X^{4}\right] -3.
\end{equation*}

\section{Parameter estimation for stationary Gaussian processes\label%
{PARAMESTIM}}

\subsection{Notation and basic question\label{NBQ}}

Consider a centered stationary Gaussian process $Z=\left( Z_{k}\right)
_{k\in \mathbb{Z}}$ with covariance
\begin{equation*}
r_{Z}(k):=E(Z_{0}Z_{k})\mbox{ for every }k\in \mathbb{Z}\mbox{ such that }%
r_{Z}(0)>0.
\end{equation*}%
For any centered Gaussian sequence $Z$ indexed by $\mathbb{Z}$ (stationary
or not), it is always possible to represent the entirely family of $Z_{n}$'s
jointly as Wiener integrals using a corresponding family of functions $%
f_{n}\in \mathcal{H}$ as
\begin{equation*}
Z_{n}=I_{1}\left( f_{n}\right)
\end{equation*}%
in the notation of Section \ref{Wiener}. In all that follows, we will use
this representation.

Fix a polynomial function $f_{q}$ where $q$ is an even integer such that $%
f_{q}$ possesses the following decomposition
\begin{equation}
f_{q}(x):=\sum_{k=0}^{q/2}d_{f_{q},2k}H_{2k}\left( \frac{x}{\sqrt{r_{Z}(0)}}%
\right)  \label{polynomial function}
\end{equation}%
where for every $k=1,\ldots ,\frac{q}{2}$, $d_{f_{q},2k}\in {\mathbb{R}}$
with $d_{f_{q},q}=r_{Z}^{\frac{q}{2}}(0)$. Thus from Section \ref{Wiener},
we can write for every $i\geq 0$
\begin{equation}
f_{q}(Z_{i})=\sum_{k=0}^{q/2}d_{f_{q},2k}H_{2k}\left( \frac{Z_{i}}{\sqrt{%
r_{Z}(0)}}\right) =\sum_{k=0}^{q/2}d_{f_{q},2k}I_{2k}\left( \varepsilon
_{i}^{\otimes 2k}\right)  \label{representation f_qZ}
\end{equation}%
with $Z_{i}~/\sqrt{r_{Z}(0)}=I_{1}(\varepsilon _{i})$. Define the following
partial sum
\begin{equation*}
Q_{f_{q},n}(Z):=\frac{1}{n}\sum_{i=0}^{n-1}f_{q}(Z_{i})
\end{equation*}%
and
\begin{equation*}
\lambda _{f_{q}}(Z):=E\left[ f_{q}(Z_{0})\right] .
\end{equation*}%
We expect that the polynomial variation $Q_{f_{q},n}(Z)$, as an empirical
mean, should converge to $\lambda _{f_{q}}(Z)$. Our aim in this section is
to estimate the parameter $\lambda _{f_{q}}(Z)$ and the speed of convergence
of $Q_{f_{q},n}(Z)$ to it.

The \textquotedblleft quadratic\textquotedblright\ case $q=2$ is of special
importance. In this case, the quadratic function $f_{2}$ will typically be
taken as%
\begin{equation}
f_{2}\left( x\right) =x^{2}=r_{Z}\left( 0\right) +r_{Z}\left( 0\right)
H_{2}\left( \frac{x}{\sqrt{r_{Z}\left( 0\right) }}\right) ,  \label{f2}
\end{equation}%
and may also be taken as $H_{2}\left( x\right) =x^{2}-1$ when convenient. We
will see in Theorem \ref{CLT} that certain functionals related to the
quadratic case control the estimator's asymptotics no matter what $q$ is. We
will also provide an optimal treatment in the case $q=2$ itself in Section %
\ref{QC}.

\subsection{Consistency\label{Consist}}

\begin{theorem}
\label{consistency for Z}Suppose that there exists $\varepsilon >0$ such
that for every $n\geq 0$
\begin{equation}
\sum_{k=0}^{n-1}r_{Z}(j)^{2}\leqslant n^{1-\varepsilon }.
\label{condition of consistency}
\end{equation}%
Then $Q_{f_{q},n}(Z)$ is a consistent estimator of $\lambda _{f_{q}}(Z)$,
i.e. almost surely as $n\rightarrow \infty $,%
\begin{equation}
Q_{f_{q},n}(Z)\longrightarrow \lambda _{f_{q}}(Z).  \label{cv of consistency}
\end{equation}
\end{theorem}

\begin{proof}
It follows from (\ref{representation f_qZ}) that
\begin{eqnarray}
&&E\left[ \left( Q_{f_{q},n}(Z)-\lambda _{f_{q}}(Z)\right) ^{2}\right] =E%
\left[ \left( \frac{1}{n}\sum_{j=0}^{n-1}f_{q}(Z_{j})-Ef_{q}(Z_{j})\right)
^{2}\right] =\sum_{k=1}^{q/2}d_{f_{q},2k}^{2}\frac{(2k)!}{n^{2}}%
\sum_{i,j=0}^{n-1}\left( \frac{E(Z_{i}Z_{j})}{r_{Z}(0)}\right) ^{2k}  \notag
\\
&=&\sum_{k=1}^{q/2}d_{f_{q},2k}^{2}\frac{(2k)!}{n^{2}}\sum_{i,j=0}^{n-1}%
\left( \frac{r_{Z}(i-j)}{r_{Z}(0)}\right)
^{2k}=\sum_{k=1}^{q/2}d_{f_{q},2k}^{2}\frac{(2k)!}{n}\left( 1+\frac{2}{n}%
\sum_{j=1}^{n-1}(n-j)\left( \frac{r_{Z}(j)}{r_{Z}(0)}\right) ^{2k}\right)
\notag \\
&=&\sum_{k=1}^{q/2}d_{f_{q},2k}^{2}\frac{(2k)!}{n}\left(
1+2\sum_{j=1}^{n-1}\left( \frac{r_{Z}(j)}{r_{Z}(0)}\right) ^{2k}-\frac{2}{n}%
\sum_{j=1}^{n-1}j\left( \frac{r_{Z}(j)}{r_{Z}(0)}\right) ^{2k}\right) .
\label{variance of Q_{f_q,n}(Z)}
\end{eqnarray}%
Now, using (\ref{variance of Q_{f_q,n}(Z)}), (\ref{condition of consistency}%
), (\ref{hypercontractivity}) and Lemma \ref{Borel-Cantelli} in the
appendix, the convergence (\ref{cv of consistency}) is obtained.
\end{proof}

\begin{remark}
If $Z$ is ergodic, the convergence (\ref{cv of consistency}) is immediate.
\end{remark}

\subsection{Asymptotic distribution\label{AD}}

Consider the following renormalized partial sum
\begin{equation}
U_{f_{q},n}(Z)=\sqrt{n}\left( Q_{f_{q},n}(Z)-\lambda _{f_{q}}(Z)\right)
\label{defUZ}
\end{equation}%
Then, by (\ref{representation f_qZ}) we can write
\begin{equation}
U_{f_{q},n}(Z)=\sum_{k=1}^{q/2}I_{2k}(g_{2k,n})
\label{expression of U_{f_q,n}}
\end{equation}%
where
\begin{equation*}
g_{2k,n}:=d_{f_{q},2k}\frac{1}{\sqrt{n}}\sum_{i=0}^{n-1}\varepsilon
_{i}^{\otimes 2k}.
\end{equation*}%
The following condition will play an important role in our analysis:
\begin{equation}
u_{f_{2}}\left( Z\right) :=2\sum_{j\in \mathbb{Z}}r_{Z}(j)^{2}<\infty .
\label{BM}
\end{equation}%
Under this condition, we can pursue the analysis of expression further.
Under (\ref{BM}), $r_{Z}(j)^{2}$ must converge to $0$ as $\left\vert
j\right\vert \rightarrow \infty $. Therefore, under (\ref{BM}), for any $k$,
$r_{Z}\left( j\right) ^{2k}$ is dominated by $r_{Z}\left( j\right) ^{2}$ for
large $\left\vert j\right\vert $, and the last term in (\ref{variance of
Q_{f_q,n}(Z)}) can be estimated as follows. We first fix an $\varepsilon \in
(0,1)$ and write for any $n\geq 2$
\begin{eqnarray*}
\frac{1}{n}\sum_{j=1}^{n-1}j\cdot r_{Z}(j)^{2k} &=&\frac{1}{n}%
\sum_{j=1}^{[\varepsilon n]-1}j\cdot r_{Z}(j)^{2k}+\frac{1}{n}%
\sum_{j=[\varepsilon n]}^{n-1}j\cdot r_{Z}(j)^{2k} \\
&\leqslant &\frac{\varepsilon n}{n}\sum_{j=1}^{[\varepsilon
n]-1}r_{Z}(j)^{2k}+\frac{n}{n}\sum_{j=[\varepsilon
n]}^{n}r_{Z}(j)^{2k}\leqslant \varepsilon ~u_{f_{2}}\left( Z\right)
+\sum_{j=[\varepsilon n]}^{\infty }r_{Z}(j)^{2}.
\end{eqnarray*}%
By Condition (\ref{BM}), with $\varepsilon $ fixed, one can choose $n$ so
large that $\sum_{j=[\varepsilon n]}^{\infty }r_{Z}(j)^{2}<\varepsilon $.
Thus the last term in (\ref{variance of Q_{f_q,n}(Z)}) can be made
arbitrarily small. This immediately implies the following useful result.

\begin{lemma}
\label{uZlemma}Under Condition (\ref{BM}), for every even $q\geq 2$,%
\begin{equation}
u_{f_{q}}(Z):=\lim_{n\rightarrow \infty }E\left[ U_{f_{q},n}^{2}(Z)\right]
=\sum_{k=1}^{q/2}d_{f_{q},2k}^{2}(2k)!\sum_{j\in \mathbb{Z}}\left( \frac{%
r_{Z}(j)}{r_{Z}(0)}\right) ^{2k}<\infty .  \label{limit
u_{f_q}}
\end{equation}
\end{lemma}

The following notation will be convenient.%
\begin{equation}
F_{f_{q},n}(Z):=\frac{U_{f_{q},n}(Z)}{\sqrt{E\left[ U_{f_{q},n}^{2}(Z)\right]
}}.  \label{defFZ}
\end{equation}%
When the expansion (\ref{polynomial function}) defining the polynomial $%
f_{q} $ has more than one term, we establish the following general central
limit theorem for $F_{f_{q},n}(Z)$ with explicit speed of convergence in
total variation.

\begin{theorem}
\label{CLT} Let $f_{q}$ be the function defined in (\ref{polynomial function}%
), and recall the stationary Gaussian process $Z$ with covariance function $%
r_{Z}$ on $\mathbb{Z}$, the partial sum $Q_{f_{q},n}(Z):=\frac{1}{n}%
\sum_{i=0}^{n-1}f_{q}(Z_{i})$, its renormalized version $U_{f_{q},n}(Z)$
defined in (\ref{defUZ}), and its standardized version $F_{f_{q},n}(Z)$ in (%
\ref{defFZ}). Denote $N\sim \mathcal{N}(0,1)$. Then there exists a constant $%
C_{q}(Z)$ depending on $q,f_{q}$ and $r_{Z}(0)$ such that
\begin{equation*}
d_{TV}\left( F_{f_{q},n}(Z),N\right) \leqslant C_{q}(Z)\sqrt{\sqrt{\kappa
_{4}(F_{f_{2},n}(Z))}+\kappa _{4}(F_{f_{2},n}(Z))}.
\end{equation*}%
Thus $F_{f_{q},n}(Z)$ is asymptotically normal as soon as $\kappa
_{4}(F_{f_{2},n}(Z))\rightarrow 0$. In addition, for large $n$,
\begin{equation*}
\kappa _{4}(F_{f_{2},n}(Z))=\frac{\kappa _{4}(U_{f_{2},n}(Z))}{\left( E\left[
U_{f_{2},n}^{2}(Z)\right] \right) ^{2}}=\mathcal{O}\left( \frac{\left(
\sum_{|j|<n}\left\vert r_{Z}(j)\right\vert ^{4/3}\right) ^{3}}{n\left( E%
\left[ U_{f_{2},n}^{2}(Z)\right] \right) ^{2}}\right) .
\end{equation*}
\end{theorem}

\begin{proof}
See Appendix.
\end{proof}

The upper bound in the previous theorem does not require normal convergence,
and even when this convergence holds, it does not require that the variance $%
E[U_{f_{q},n}^{2}(Z)]$ be bounded. By Lemma \ref{uZlemma}, this boundedness
holds if and only if Condition (\ref{BM}) holds.

In the next corollary, we look at two examples, one under Condition (\ref{BM}%
) and one when it fails but normality still holds. In the former case, we
replace the normalization term $\sqrt{E\left[ U_{f_{q},n}^{2}(Z)\right] }$%
which is an unobservable sequence because it depends on the
parameter-dependent sequence $r_{Z}$, by the constant $\sqrt{u_{f_{q}}(Z)}$.
While this constant also depends on the parameter $\lambda _{f_{q}}(Z)$, it
allows one to measure the total variation distance of the data-based
estimator $U_{f_{q},n}(Z)$ itself to the fixed law $\mathcal{N}\left(
0,u_{f_{q}}(Z)\right) $, consistent with common methodological practice.
This change of normalization results in an additional term to reflect the
speed of convergence of $\sqrt{E\left[ U_{f_{q},n}^{2}(Z)\right] }$ to $%
\sqrt{u_{f_{q}}(Z)}$.

\begin{corollary}
\label{CLTCor}1) If $r_{Y}(k)\sim ck^{-\frac{1}{2}}$,then
\begin{equation}
E\left[ U_{f_{q},n}^{2}(Z)\right] \sim 4c^{2}d_{q,2}^{2}(Z)\log (n).
\label{logequiv}
\end{equation}%
and the upper bound on $d_{TV}\left( F_{f_{q},n}(Z),N\right) $ from Theorem %
\ref{CLT} holds with
\begin{equation}
\kappa _{4}(F_{f_{2},n}(Z))=\mathcal{O}\left( \log ^{-2}(n)\right) .
\label{logkappa4}
\end{equation}%
2) Under Condition (\ref{BM}), i.e. if $\sum_{j\in \mathbb{Z}}\left\vert
r_{Z}(j)\right\vert ^{2}<\infty $, we have
\begin{eqnarray}
&&d_{TV}\left( U_{f_{q},n}(Z),\mathcal{N}\left( 0,u_{f_{q}}(Z)\right) \right)
\notag \\
&\leqslant &C_{q}(Z)\sqrt{\sqrt{\kappa _{4}(F_{f_{2},n}(Z))}+\kappa
_{4}(F_{f_{2},n}(Z))}+2\left\vert 1-\frac{E\left[ U_{f_{q},n}^{2}(Z)\right]
}{u_{f_{q}}(Z)}\right\vert  \label{d_{TV}(F_{q,n}(Z),N)finite} \\
&\leqslant &C_{q}(Z)\sqrt{\sqrt{\frac{\kappa _{4}(U_{f_{2},n}(Z))}{%
\left\vert u_{f_{q}}(Z)\right\vert ^{2}}}+\frac{\kappa _{4}(U_{f_{2},n}(Z))}{%
\left\vert u_{f_{q}}(Z)\right\vert ^{2}}}+2\left\vert 1-\frac{E\left[
U_{f_{q},n}^{2}(Z)\right] }{u_{f_{q}}(Z)}\right\vert .  \label{CorPoint2b}
\end{eqnarray}

3) Under the additional assumption that $r_{Z}$ is asymptotically of
constant sign and monotone, the expressions in (\ref%
{d_{TV}(F_{q,n}(Z),N)finite}) and (\ref{CorPoint2b}) converges to $0$.
\end{corollary}

\begin{proof}
The estimate (\ref{logequiv}) is a direct consequence of (\ref{variance of
Q_{f_q,n}(Z)}). Also, by (\ref{logequiv}) and the second estimate of Theorem %
\ref{CLT} we obtain (\ref{logkappa4}). The result of point (1) is
established.

Next, we prove the estimate (\ref{d_{TV}(F_{q,n}(Z),N)finite}). We first
note that since the laws of $U_{f_{q},n}(Z)$ and of $N$ both have densities
with respect to Lebesgue's measure, $d_{TV}\left( U_{f_{q},n}(Z),\sqrt{%
u_{f_{q}}(Z)}N\right) $ is identical to $d_{TV}\left( U_{f_{q},n}(Z)/\sqrt{%
u_{f_{q}}(Z)},N\right) $. Next, from \cite[Proposition 2.4]{NP2013} (see
bullet points in Section \ref{Wiener}) we can write
\begin{eqnarray*}
&&d_{TV}\left( \frac{U_{f_{q},n}(Z)}{\sqrt{u_{f_{q}}(Z)}},N\right) \\
&\leqslant &2E\left\vert 1-\left\langle D\frac{U_{f_{q},n}(Z)}{\sqrt{%
u_{f_{q}}(Z)}},-DL^{-1}\frac{U_{f_{q},n}(Z)}{\sqrt{u_{f_{q}}(Z)}}%
\right\rangle _{\mathcal{H}}\right\vert \\
&\leqslant &2\left\vert 1-\frac{E\left[ U_{f_{q},n}^{2}(Z)\right] }{%
u_{f_{q}}(Z)}\right\vert +2\frac{E\left[ U_{f_{q},n}^{2}(Z)\right] }{%
u_{f_{q}}(Z)}E\left\vert 1-\left\langle D\frac{U_{f_{q},n}(Z)}{\sqrt{E\left[
U_{f_{q},n}^{2}(Z)\right] }},-DL^{-1}\frac{U_{f_{q},n}(Z)}{\sqrt{E\left[
U_{f_{q},n}^{2}(Z)\right] }}\right\rangle _{\mathcal{H}}\right\vert
\end{eqnarray*}%
By the expression in (\ref{variance of Q_{f_q,n}(Z)}), the ratio $E\left[
U_{f_{q},n}^{2}(Z)\right] /u_{f_{q}}(Z)$ is in $(0,1)$. Therefore the first
estimate in point (2) follows by the main estimate in the proof of Theorem %
\ref{CLT}. The estimate (\ref{CorPoint2b}) is an elementary consequence of
the same fact that $E\left[ U_{f_{q},n}^{2}(Z)\right] <u_{f_{q}}(Z)$. Point
(2) is thus fully established.

To prove the corollary's final claim in point (3), we first note that by
Lemma \ref{uZlemma}, the term $\left\vert 1-E\left[ U_{f_{q},n}^{2}(Z)\right]
/u_{f_{q}}(Z)\right\vert $ tends to $0$. Thus we only need to show that
under Condition (\ref{BM}) and the additional monotonicity assumption on $%
r_{Z}$, $\kappa _{4}(U_{f_{2},n}(Z))$ also tends to $0$. By the conclusion
of Theorem \ref{CLT} and the finiteness of $u_{f_{2}}\left( Z\right) $, we
have%
\begin{equation}
\kappa _{4}(U_{f_{2},n}(Z))=\mathcal{O}\left( n^{-1}\left(
\sum_{|j|<n}\left\vert r_{Z}(j)\right\vert ^{4/3}\right) ^{3}\right)
=:K_{4}\left( n\right) .  \label{K4}
\end{equation}%
Next, we borrow from \cite[Proposition 1]{NV2014} that, under the additional
assumptions on $r_{Z}$ in the last statement of the theorem, and using the
finiteness of $u_{f_{2}}\left( Z\right) $,
\begin{equation*}
K_{4}\left( n\right) =\mathcal{O}\left( n^{-1/3}\left( n^{-1/2}\left(
\sum_{|j|<n}\left\vert r_{Z}(j)\right\vert ^{3/2}\right) ^{2}\right)
^{4/3}\right) .
\end{equation*}%
From Jensen's inequality and Lemma \ref{uZlemma} we get
\begin{equation*}
n^{-1/2}\left( \sum_{|j|<n}\left\vert r_{Z}(j)\right\vert ^{3/2}\right)
^{2}\leqslant 4n^{3/2}\left( \frac{1}{2n}\sum_{|j|<n}\left\vert
r_{Z}(j)\right\vert ^{2}\right) ^{3/2}\leqslant \sqrt{2}u_{f_{2}}\left(
Z\right) <\infty .
\end{equation*}%
Thus by (\ref{K4}), $\kappa _{4}(U_{f_{2},n}(Z))=\mathcal{O}\left(
n^{-1/3}\right) $, which finishes the proof of the corollary.
\end{proof}

The assumptions of Corollary \ref{CLTCor} are very weak, given a memory
length shorter than that of fractional Gaussian noise with Hurst parameter $%
H=5/8$. We state this formally in the next Corollary, while leaving the
question of how to rid oneself of the condition \textquotedblleft $H<5/8$%
\textquotedblright\ to Section \ref{IRC}. In fact, the computations outlined
in Section \ref{IRC}, which are based on the convergence speeds (\ref%
{variancespeed}) and (\ref{kappa4speed}) identified in Section \ref{EXGEN}
which immediately follows, are sufficient to obtain this corollary, whose
simple proof we thus omit. Note however that the corollary's result is not
necessarily sharp; we will see in Section \ref{QC} that it is not sharp for $%
q=2$.

\begin{corollary}
\label{CLTCor2}Under the notation of Theorem \ref{CLT}, assume that for some
$H<5/8$,
\begin{equation*}
r_{Z}\left( k\right) \leqslant ck^{2H-2}.
\end{equation*}%
Then Condition (\ref{BM}) holds and for some constant $C$ depending only on $%
q$ and $r_{Z}\left( 0\right) ,$
\begin{equation*}
d_{TV}\left( U_{f_{q},n}(Z),\mathcal{N}\left( 0,u_{f_{q}}(Z)\right) \right)
\leqslant Cn^{-1/4}.
\end{equation*}
\end{corollary}

\begin{remark}
\label{NCLT}In most cases where Condition (\ref{BM}) fails, the series'
divergence occurs so fast that $U_{f_{q},n}(Z)$'s asymptotics are not
normal. Under certain special circumstances, namely a slowly modulated $%
(2H-2)$-self-similarity assumption on $r_{Z}$, classical tools such as in
\cite{DM} can be used to show that $U_{f_{q},n}(Z)$ tends to a so-called
(scaled) Rosenblatt law $G_{\infty }^{\left( H\right) }$. A now classical
result of Davydov and Martinova \cite{DaMa} was revived in recent years in
\cite{BN, NV2014} to estimate total-variation distances to $G_{\infty
}^{\left( H\right) }$. This can be achieved in our context as well, though
for the sake of conciseness, we omit this study, only stating two basic
results here, whose proofs would proceed as in \cite{NV2014} and \cite{BN}
respectively.

\begin{enumerate}
\item Assume that for some $H\in (3/4,1)$ and some $\beta >0$,
asymptotically
\begin{equation*}
\left( 1+o\left( 1\right) \right) \log ^{-\beta }\left( \left\vert
k\right\vert \right) \left\vert k\right\vert ^{2H-2}\leqslant \left\vert
r_{Z}(k)\right\vert /r_{Z}\left( 0\right) \leqslant \left( 1+o\left(
1\right) \right) \log ^{\beta }\left( \left\vert k\right\vert \right)
\left\vert k\right\vert ^{2H-2},
\end{equation*}%
then for some constant $C$ depending on $r$ and $H$,
\begin{equation*}
d_{W}\left( \frac{U_{f_{2},n}(Z)}{2\sum_{\left\vert k\right\vert
>n}r_{Z}(k)^{2}},\frac{1}{2}\sqrt{\frac{4H-3}{2\Gamma (2-2H)\cos \left(
\frac{(2-2H)\pi }{2}\right) }}G_{\infty }^{\left( H\right) }\right)
\leqslant \frac{C}{\sqrt{\log n}}.
\end{equation*}

\item If $\beta =0$, then $f_{2}$ can be replaced above by $f_{q}$ for any
even $q$, and $3/4$ can be replaced by $1-1/(2q)$, and $\sqrt{\log n}$ by $%
n^{H-1+1/(2q)}$.
\end{enumerate}

The law of $G_{\infty }^{\left( H\right) }$ can be represented under a
standard white noise measure $W$ on $\mathbb{C}$ as
\begin{equation*}
G_{\infty }^{\left( H\right) }=\int \int_{{\mathbb{R}}^{2}}e^{i(x+y)}\frac{%
e^{i(x+y)}-1}{i(x+y)}|xy|^{1/2-H}W(dx)W(dy).
\end{equation*}
\end{remark}

\subsection{Examples\label{EX}}

\subsubsection{General considerations\label{EXGEN}}

For any degree-$q$ polynomial $f_{q}$ as given in (\ref{polynomial function}%
), we saw that under the assumptions of Theorem \ref{consistency for Z} and
Theorem \ref{CLT}, $Q_{f_{q},n}(Z)$ serves as a consistent and
asymptotically normal generalized method of moments estimator $\lambda
_{f_{q}}(Z):=E\left[ f_{q}(Z_{0})\right] $, or indeed for any parameter
which can be extracted from this quantity.

Condition (\ref{BM}) should be sought in order to invoke the explicit speed
of convergence result of the second part of Corollary \ref{CLTCor}. This
assumption is generic for any stationary process whose memory length is
bounded above strictly by that of the so-called fractional Gaussian noise
with Hurst parameter $H=3/4$. The articles \cite{NV2014} and \cite{NP2013}
can be consulted for precise statements of what this means in general scales
and in the fractional Brownian scale. Section \ref{IRC} can be consulted for
a simple transformation of the data to ensure that Condition (\ref{BM})
holds for any long-memory stationary Gaussian sequence with an
asymptotically power-law autocorrelation decay.

Under Condition (\ref{BM}), Part (2) of Corollary \ref{CLTCor} shows that
the speed of convergence in total variation is determined by the choice of $%
q $ via the leading constant $C_{q}\left( Z\right) $ and the speed of
convergence of the variance term $E\left[ U_{f_{q},n}^{2}(Z)\right] $.
Specifically, by Lemma \ref{uZlemma}, the term corresponding to this
variance convergence, i.e. the last term in (\ref{CorPoint2b}), is given by
the tail expression%
\begin{equation}
\frac{2}{u_{f_{q}}(Z)}\sum_{k=1}^{q/2}d_{f_{q},2k}^{2}\left( Z\right)
(2k)!\sum_{\left\vert j\right\vert >n}\left( \frac{r_{Z}(j)}{r_{Z}(0)}%
\right) ^{2k}.  \label{variancespeed}
\end{equation}%
The other term in (\ref{CorPoint2b}), which corresponds to the core normal
convergence from Theorem \ref{CLT}, asymptotically equivalent to
\begin{equation}
\frac{C_{q}(Z)}{\sqrt{u_{f_{q}}(Z)}}\left( \kappa
_{4}(U_{f_{2},n}(Z))\right) ^{1/4};  \label{kappa4speed}
\end{equation}%
this has a speed which is determined by the fourth root of the estimator's
fourth cumulant no matter what $q$ is, and the leading constant $C_{q}\left(
Z\right) $ can be computed via the explicit formulas (\ref{C1q}), (\ref{C2q}%
), and (\ref{Cq}) in the proof of Theorem \ref{CLT}.

There may be a trade-off between choosing a large $q$ to effect the size of
the coefficients $d_{f_{q},2k}^{2}$ and a small $q$ to control the value of $%
C_{q}\left( Z\right) $. The constants in these expressions are sufficiently
complex to make it difficult to discern a general rule on how to choose $q$,
particularly since the speed of convergence of $E\left[ U_{f_{q},n}^{2}(Z)%
\right] $ depends heavily on the entire sequence $r_{Z}$. But this can be
determined on a case-by-case basis since all the constants can be computed
explicitly, as the examples in the subsections that follow show.

Before working out those examples, we finish this section with an attempt to
explain in qualitative terms where the trade-off may come from. For the sake
of argument, let us compare the constants for $q=2$ with those for a large $%
q $.

For $q=2$, the constant $C_{2}\left( Z\right) $ is of moderate size:
inspection of the proof of Theorem \ref{CLT} shows that $C_{2}\left(
Z\right) =2\sqrt{2}r_{Z}\left( 0\right) $, while the variance-convergence
term in (\ref{variancespeed}) easily computes to%
\begin{equation*}
\frac{2}{u_{f_{2}}(Z)}\sum_{\left\vert j\right\vert >n}r_{Z}(j)^{2}=\frac{%
\sum_{\left\vert j\right\vert >n}r_{Z}(j)^{2}}{\sum_{j\in \mathbb{Z}%
}r_{Z}(j)^{2}}.
\end{equation*}

One may choose a high-degree polynomial $f_{q}$ such that the constant $%
d_{f_{q},2}^{2}\left( Z\right) $ may be much smaller than $C_{2}\left(
Z\right) $; in this case, the dominant term in (\ref{variancespeed}), which
is for $k=1$, has the same behavior in terms of $n$ as for $q=2$, but would
be minimized because of the presence of the small multiplicative factor $%
d_{f_{q},2}^{2}\left( Z\right) $. If the other constants $%
d_{f_{q},2k}^{2}\left( Z\right) $ for $k\geq 2$ were much larger, this would
have little effect for large $n$ since they would be multiplicative of the
asymptotically negligible tails $\sum_{\left\vert j\right\vert
>n}r_{Z}(j)^{2k}$ in (\ref{variancespeed}). In other words, for large $q$,
the speed of convergence in (\ref{variancespeed}) can be controlled by
choosing $f_{q}$ with a small contribution to the term corresponding to $%
H_{2}$ in the Hermite polynomial decomposition (\ref{polynomial function}).
However, this must be traded off against the size of the constant $%
C_{q}\left( Z\right) $. All the terms in the expression $C_{q}\left(
Z\right) $ are additive and grow quickly as $q$ increases; the term of
highest order is proportional to $q^{3/2}\left( 2q-4\right) !$. Such rapid
growth does not seem to be the case for $d_{f_{q},2}^{2}\left( Z\right) $,
as illustrated in the next two typical examples, where access to the
variance parameter $E\left( Z_{0}^{2}\right) =r_{Z}\left( 0\right) $ itself
is essentially immediate.

\subsubsection{Hermite variation}

Let $q\in \mathbb{N}^{\ast }$ be an even integer. Then the $q$th Hermite
polynomial $H_{q}$ can be written as in (\ref{polynomial function}). Indeed,
it follows from the fact that for $q$ even, $H_{q}(x)=\sum_{k=0}^{\frac{q}{2}%
}\frac{q!(-1)^{k}}{k!(q-2k)!2^{k}}x^{q-2k}$. Then we can write,
\begin{equation*}
H_{q}(x)=EH_{q}(Z_{0})+\sum_{k=1}^{q/2}d_{H_{q},2k}\left( Z\right)
H_{2k}\left( \frac{x}{\sqrt{r_{Z}(0)}}\right)
\end{equation*}%
where for any $k\in \{1,\ldots ,\frac{q}{2}-1\}$
\begin{eqnarray*}
d_{H_{q},q-2k}(Z) &=&(-1)^{k}\left( r_{Z}^{\frac{q}{2}}(0)-r_{Z}^{\frac{q}{2}%
-1}(0)\right) a_{q-2}^{q}a_{q-4}^{q-2}\dots a_{q-2k}^{q-2k+2} \\
+ &&(-1)^{k-1}\left( r_{Z}^{\frac{q}{2}}(0)-r_{Z}^{\frac{q}{2}-2}(0)\right)
a_{q-4}^{q}a_{q-6}^{q-4}\ldots a_{q-2k}^{q-2k+2} \\
+ &&\ldots \\
+ &&(-1)^{1}\left( r_{Z}^{\frac{q}{2}}(0)-r_{Z}^{\frac{q}{2}-k}(0)\right)
a_{q-2k}^{q}
\end{eqnarray*}%
and $d_{H_{q},q}(Z)=r_{Z}^{\frac{q}{2}}(0)$, where for every $p$ even, the
constants
\begin{equation*}
a_{p-2k}^{p}=\frac{p!(-1)^{k}}{k!(p-2k)!2^{k}}\quad k=0,\ldots ,p/2
\end{equation*}%
are the ones which satisfy
\begin{equation*}
H_{p}(x)=\sum_{k=0}^{p/2}a_{p-2k}^{p}x^{p-2k}.
\end{equation*}%
Consequently, the Hermite variation
\begin{equation*}
Q_{H_{q},n}(Z):=\frac{1}{n}\sum_{k=0}^{n-1}H_{q}(Z_{k})
\end{equation*}%
satisfies the results given in Sections \ref{Consist} and \ref{AD}. Moreover
the parameter $\lambda _{H_{q}}(Z)$ has the following explicit expression
\begin{eqnarray}
\lambda _{H_{q}}(Z) &=&E\left[ H_{q}(Z_{0})\right] =\sum_{k=0}^{\frac{q}{2}}%
\frac{q!(-1)^{k}}{k!(q-2k)!2^{k}}E\left( Z_{0}^{q-2k}\right)  \notag \\
&=&\frac{q!}{2^{q/2}}\sum_{k=0}^{\frac{q}{2}}\frac{(-1)^{k}}{k!(\frac{q}{2}%
-k)!}\left[ E\left( Z_{0}^{2}\right) \right] ^{\frac{q}{2}-k}  \notag \\
&=&\frac{q!}{(\frac{q}{2})!2^{q/2}}\left( E\left( Z_{0}^{2}\right) -1\right)
^{q/2}.  \label{alphaHqcase}
\end{eqnarray}

Thus the results of the previous sections provide explicit means for
computing total variation speeds of convergence in a generalized method of
moments based on Hermite polynomials for estimating the variance parameter $%
E\left( Z_{0}^{2}\right) =r_{Z}\left( 0\right) $. Because of the simple form
of (\ref{alphaHqcase}) as a function of $r_{Z}\left( 0\right) $, for any
sequence satisfying Condition (\ref{BM}), one can immediately test the
hypothesis of whether $r_{Z}\left( 0\right) $ equals a specific value $%
\sigma ^{2}$, using Corollary \ref{CLTCor} to account precisely for the
error term due to non-infinite sample size. Since the corollary provides the
error in total variation distance, this error is uniform over $\sigma ^{2}$
by definition.

\subsubsection{Power variation}

Let $q\in \mathbb{N}^{\ast }$ be an even integer. Let $c_{q,2k}=\frac{1}{%
(2k)!}\int_{-\infty }^{\infty }\frac{e^{-x^{2}/2}}{\sqrt{2\pi }}%
x^{q}H_{2k}(x)\ dx$ be the coefficients of the monomial $\phi _{q}(x):=x^{q}$
expanded in the basis of Hermite polynomials :
\begin{equation*}
x^{q}=\sum_{k=0}^{q/2}c_{q,2k}H_{2k}(x).
\end{equation*}%
It is known that
\begin{equation*}
c_{q,2k}=\frac{q!}{2^{q/2-k}\left( q/2-k\right) !\left( 2k\right) !}.
\end{equation*}%
Thus, by relying directly on the results we just saw in the case of Hermite
polynomials, the polynomial function $\phi _{q}$ can be written as in (\ref%
{polynomial function}). As consequence, the power variation
\begin{equation*}
Q_{\phi _{q},n}(Z)=\frac{1}{n}\sum_{i=0}^{n-1}(Z_{i})^{q}.
\end{equation*}%
satisfies the results given in Sections \ref{Consist} and \ref{AD}. In this
case, the parameter $\lambda _{\phi _{q}}(Z)$ has the following explicit
expression
\begin{equation}
\lambda _{\phi _{q}}(Z)=E\left[ (Z_{0})^{q}\right] =\frac{q!}{(\frac{q}{2}%
)!2^{q/2}}\left[ E\left( Z_{0}^{2}\right) \right] ^{q/2}.
\label{alphaPhicase}
\end{equation}

\section{Parameter estimation for non-stationary Gaussian process\label%
{NONSTAT}}

In practice, it is often the case that the data comes from a sequence which
has visibly not yet reached a stationary regime. This is a typical situation
for the solution of a stochastic system which initiates from a point mass
rather than the system's stationary distribution; we will see examples of
this in Sections \ref{AOUP} and \ref{Multi}. The rate at which stationarity
is reached heavily affects other rates of convergence, including the total
variation speeds in the central limit theorem. To illustrate this phenomenon
more broadly than in the two aforementioned sections, in this section we
consider a general class of models which can be written as the sum of a
stationary model and a non-stationary nuisance term which vanishes
asymptotically.

\subsection{General case}

For a polynomial $f_{q}$ of even degree $q$, and a random sequence $X$,
recall the polynomial variation notation introduced in Section \ref{NBQ}:%
\begin{equation*}
Q_{f_{q},n}(X):=\frac{1}{n}\sum_{i=0}^{n-1}f_{q}(X_{i}).
\end{equation*}%
Let $\left( Z_{k}\right) _{k\in \mathbb{Z}}$ be a centered stationary
Gaussian process and let $\left( Y_{k}\right) _{k\in \mathbb{Z}}$ be a
process such that the following condition holds: there exist $p_{0}\in
\mathbb{N}$ and a constant $\gamma >0$ such that for every $p\geq p_{0}$ and
for all $n\in \mathbb{N}$,
\begin{equation}
\left\Vert Q_{f_{q},n}(Z+Y)-Q_{f_{q},n}(Z)\right\Vert _{L^{p}(\Omega )}=%
\mathcal{O}\left( n^{-\gamma }\right) .  \label{hypothesis on Z+Y}
\end{equation}%
Combining (\ref{hypothesis on Z+Y}), Lemma \ref{Borel-Cantelli} in the
Appendix, and Theorem \ref{consistency for Z} we get the following result.

\begin{theorem}
\label{consistency for Z+Y}Assume that the conditions (\ref{hypothesis on
Z+Y}) and (\ref{condition of consistency}) hold. Then
\begin{equation*}
Q_{f_{q},n}(Z+Y)\longrightarrow \lambda _{f_{q}}(Z)
\end{equation*}%
almost surely as $n\rightarrow \infty $.
\end{theorem}

In Corollary \ref{CLTCor}, we handled a discrepancy at the level of
deterministic normalizing constants, while retaining statements with the
total variation distance. In this section, our discrepancy comes at a
slightly higher price because it is stochastic. We use instead the
Wasserstein distance $d_{W}$, in order to rely on the following elementary
lemma whose proof is in the Appendix.

\begin{lemma}
\label{dWL1}Let $Y$ and $Z$ be random variables defined on the same
probability space. Then
\begin{equation*}
d_{W}\left( Y+Z,N\right) \leqslant d_{W}\left( Z,N\right) +\left\Vert
Y\right\Vert _{L^{1}\left( \Omega \right) }.
\end{equation*}
\end{lemma}

Another lemma, proved for instance in \cite[Theorem 5.1.3]{NP-book}, relates
the Wasserstein distance to a connection between Stein's method and the
Malliavin calculus.

\begin{lemma}
If $F$ has mean $0$, variance $1$, and a square-integrable Malliavin
derivative, then
\begin{equation*}
d_{W}\left( F,N\right) \leqslant \sqrt{\frac{2}{\pi }}E\left[ \left\vert
1-\left\langle DF,-DL^{-1}F\right\rangle \right\vert \right]
\end{equation*}
\end{lemma}

By combining these two lemmas (the second one applies because variables with
finite chaos expansions are infinitely Malliavin-differentiable with finite
moments of all orders) and the proof of Theorem \ref{CLT}, by (\ref%
{hypothesis on Z+Y}) we immediately obtain the following upper bounds.

\begin{theorem}
\label{CLT for Z+Y}Under hypothesis (\ref{hypothesis on Z+Y}) and the
assumptions of Theorem \ref{CLT}, for some constant $C$ depending on the
relation in (\ref{hypothesis on Z+Y}),
\begin{equation*}
d_{W}\left( \frac{U_{f_{q},n}(Z+Y)}{\sqrt{E\left[ U_{f_{q},n}^{2}(Z)\right] }%
},N\right) \leqslant C\frac{n^{\frac{1}{2}-\gamma }}{\sqrt{E\left[
U_{f_{q},n}^{2}(Z)\right] }}+\frac{C_{q}\left( Z\right) \sqrt{2}}{\sqrt{\pi }%
}\sqrt{\sqrt{\kappa _{4}(F_{f_{2},n}(Z))}+\kappa _{4}(F_{f_{2},n}(Z))}.
\end{equation*}%
In addition, under Condition (\ref{BM}), i.e. if $\sum_{j\in \mathbb{Z}%
}\left\vert r_{Z}(j)\right\vert ^{2}<\infty $,
\begin{eqnarray}
d_{W}\left( \frac{U_{f_{q},n}(Z+Y)}{\sqrt{u_{f_{q}}(Z)}},N\right) &\leqslant
&C\frac{n^{\frac{1}{2}-\gamma }}{\sqrt{u_{f_{q}}(Z)}}+\frac{C_{q}\left(
Z\right) \sqrt{2}}{\sqrt{\pi }}\sqrt{\sqrt{\frac{\kappa _{4}(U_{f_{2},n}(Z))%
}{u_{f_{q}}(Z)^{2}}}+\frac{\kappa _{4}(U_{f_{2},n}(Z))}{u_{f_{q}}(Z)^{2}}}
\notag \\
&&+\sqrt{\frac{8}{\pi }}\left\vert 1-\frac{E\left[ U_{f_{q},n}^{2}(Z)\right]
}{u_{f_{q}}(Z)}\right\vert .  \label{d_{TV} for Z+Y with finite serie}
\end{eqnarray}
\end{theorem}

\subsection{Quadratic case\label{QC}}

In this subsection we assume that $q=2$. In this special case, consistent
with the notation in (\ref{f2}), without loss of generality up to
deterministic shifting and scaling, the only relevant polynomial of interest
is $f_{2}\left( x\right) =x^{2}$. Thus the question introduced in Section %
\ref{NBQ} is to estimate the variance $r_{Z}\left( 0\right) =\mathbf{E}\left[
\left( Z_{0}\right) ^{2}\right] $ where $Z$ is our stationary Gaussian
process. Using the notation introduced in that section, we thus have the
following expression for our normalized partial sum%
\begin{equation*}
U_{f_{2},n}(Z)=\frac{r_{Z}(0)}{\sqrt{n}}\sum_{k=0}^{n-1}H_{2}\left( \frac{%
Z_{k}}{\sqrt{r_{Z}(0)}}\right) =I_{2}\left( \frac{r_{Z}(0)}{\sqrt{n}}%
\sum_{k=0}^{n-1}\varepsilon _{k}^{\otimes 2}\right) ,
\end{equation*}%
where again $\varepsilon _{k}$ is defined by $Z_{k}~/\sqrt{r_{Z}(0)}%
=I_{1}\left( \varepsilon _{k}\right) $. Using the notation in Section \ref%
{AD}, the standardized version of $U_{f_{2},n}(Z)$ is thus
\begin{equation*}
F_{f_{2},n}\left( Z\right) =\frac{U_{f_{2},n}(Z)}{\sqrt{E\left[
U_{f_{2},n}^{2}(Z)\right] }}.
\end{equation*}%
Recall the 4th cumulant $\kappa _{4}\left( F_{f_{2},n}(Z)\right) =E\left[
F_{f_{2},n}(Z)^{4}\right] -3$, and define the third cumulant
\begin{equation*}
\kappa _{3}\left( F_{f_{2},n}(Z)\right) :=E\left[ F_{f_{2},n}(Z)^{3}\right] .
\end{equation*}%
We will apply the sharp asymptotics established in \cite{NP2013} (see bullet
points in Section \ref{Wiener}), by which a sequence of variance-one random
variables $F_{n}$ in a fixed Wiener chaos which converges in law to the
normal has total variation distance to the normal commensurate with the
maximum of its third and fourth cumulant. We will also apply an explicit
version of this theorem, due to \cite{NV2014}, tailored to quadratic
variations of stationary Gaussian processes. For positive-valued sequences $%
a $ and $b$, we will use the commensurability notation
\begin{equation*}
a_{n}\asymp b_{n}\iff 0<c:=\inf_{n}\frac{a_{n}}{b_{n}}\leqslant \sup_{n}%
\frac{a_{n}}{b_{n}}=:C<\infty
\end{equation*}%
where the extrema may be over all positive integers, or all integers
exceeding a value $n_{0}$. Our first result is the following.

\begin{proposition}
\label{CLT for quadratic var} (1) With $f_{2}\left( x\right) =x^{2}$, assume
that $\kappa _{4}(F_{f_{2},n}(Z))\longrightarrow 0$. Then
\begin{equation}
d_{TV}\left( F_{f_{2},n}(Z),N\right) \asymp \max \left\{ \kappa _{4}\left(
F_{f_{2},n}(Z)\right) ,\left\vert \kappa _{3}\left( F_{f_{2},n}(Z)\right)
\right\vert \right\} .  \label{optimal4}
\end{equation}

(2) If $r_{Z}$ is asymptotically of constant sign and monotone, then $\kappa
_{4}(F_{f_{2},n}(Z))\longrightarrow 0$ if and only if $\kappa
_{3}(F_{f_{2},n}(Z))\longrightarrow 0$, and in this case,%
\begin{equation}
d_{TV}\left( F_{f_{2},n}(Z),N\right) \asymp \left\vert \kappa _{3}\left(
F_{f_{2},n}(Z)\right) \right\vert =\left\vert E\left(
(F_{f_{2},n}(Z))^{3}\right) \right\vert ,  \label{optimal3}
\end{equation}%
and moreover,
\begin{equation*}
\left\vert E\left( (F_{f_{2},n}(Z))^{3}\right) \right\vert \asymp \frac{%
\left( \sum_{|k|<n}|r_{Z}(k)|^{3/2}\right) ^{2}}{\left(
\sum_{|k|<n}|r_{Z}(k)|^{2}\right) ^{3/2}\sqrt{n}}.
\end{equation*}
\end{proposition}

\begin{proof}
The result (\ref{optimal4}) in Point (1) is a direct consequence of the main
result in \cite{NP2013} (see also \cite{BBNP}). The statements in point (2)
come directly from \cite[Theorem 3]{NV2014}.
\end{proof}

The methods used to prove Corollary \ref{CLTCor} and Theorem \ref{CLT for
Z+Y} immediately lead from the upper bound statements in Proposition \ref%
{CLT for quadratic var} to the following corollary.

\begin{corollary}
\label{cor CLT for quadratic var}If the hypothesis (\ref{hypothesis on Z+Y})
holds, under the assumptions in part (2) of Proposition \ref{CLT for
quadratic var}, for some constant $C>0$,%
\begin{equation*}
d_{W}\left( \frac{U_{f_{2},n}(Z+Y)}{\sqrt{E\left[ U_{f_{2},n}^{2}(Z)\right] }%
},N\right) \leqslant C\left( \frac{n^{\frac{1}{2}-\gamma }}{\sqrt{E\left[
U_{f_{2},n}^{2}(Z)\right] }}+\left\vert E\left( (F_{f_{2},n}(Z))^{3}\right)
\right\vert \right) .
\end{equation*}%
In addition, if Condition (\ref{BM}) holds, i.e. $\sum_{j\in \mathbb{Z}%
}\left\vert r_{Z}(j)\right\vert ^{2}<\infty $, then%
\begin{eqnarray}
&&d_{W}\left( \frac{U_{f_{2},n}(Z+Y)}{\sqrt{u_{f_{2}}(Z)}},N\right)  \notag
\\
&\leqslant &C\left( \frac{n^{\frac{1}{2}-\gamma }}{\sqrt{u_{f_{2}}(Z)}}%
+\left\vert E\left( (F_{f_{2},n}(Z))^{3}\right) \right\vert \right) +C\frac{%
\sum_{\left\vert j\right\vert >n}\left\vert r_{Z}(j)\right\vert ^{2}}{%
u_{f_{2}}(Z)}.  \label{d_{TV} for Z+Y finite serie and q=2}
\end{eqnarray}
\end{corollary}

Unfortunately, these techniques say nothing about how to obtain lower bounds
when one adds discrepancies corresponding to the speed of convergence of the
series $\sum_{j}\left\vert r_{Z}(j)\right\vert ^{2}$, and to a
non-stationary term. We now investigate some slight strengthening of
Conditions (\ref{BM}) and (\ref{hypothesis on Z+Y}) which allow for such
lower-bound statements, starting with some elementary considerations.

From (\ref{optimal3}) and the remainder of Point (2) in Proposition \ref{CLT
for quadratic var}, there exists a constant $c_{1}\left( Z\right) $
depending only on the law of $Z$ such that
\begin{equation}
c_{1}\left( Z\right) \frac{\left( \sum_{|k|<n}|r_{Z}(k)|^{3/2}\right) ^{2}}{%
\left( \sum_{|k|<n}|r_{Z}(k)|^{2}\right) ^{3/2}\sqrt{n}}\leqslant
d_{TV}\left( F_{f_{2},n}(Z),N\right) \leqslant C_{1}\left( Z\right) \frac{%
\left( \sum_{|k|<n}|r_{Z}(k)|^{3/2}\right) ^{2}}{\left(
\sum_{|k|<n}|r_{Z}(k)|^{2}\right) ^{3/2}\sqrt{n}}.  \label{c1}
\end{equation}%
Now assume merely that (\ref{BM}) holds: $\sum_{j}\left\vert
r_{Z}(j)\right\vert ^{2}$ converges. Thus, for some constant $c_{2}^{\prime
}\left( Z\right) $ depending only on the law of $Z$,
\begin{equation}
d_{TV}\left( F_{f_{2},n}(Z),N\right) \geq \frac{c_{2}^{\prime }\left(
Z\right) }{\sqrt{n}}.  \label{27LB}
\end{equation}

In cases where $\sum |r_{Z}(k)|^{3/2}$ diverges, we evidently get a larger
lower bound than (\ref{27LB}), which would make the rest of the analysis
easier. To keep track of multiplicative constants as best we can, we define%
\begin{equation}
L\left( Z\right) :=\lim_{n\rightarrow \infty }\frac{\left(
\sum_{|k|<n}|r_{Z}(k)|^{3/2}\right) ^{2}}{\left(
\sum_{|k|<n}|r_{Z}(k)|^{2}\right) ^{3/2}},  \label{Ldef}
\end{equation}%
which exists and is positive under condition (\ref{BM}), with the
understanding that when $L\left( Z\right) $ is $+\infty $, one can and
should replace it by an arbitrarily large constant for $n$ large enough. We
will not comment on the case of diverging $L$ further. The interested reader
can work out for herself how much better the final lower bound results would
be in this case.

Thus in (\ref{27LB}), we may take $c_{2}^{\prime }\left( Z\right)
=c_{1}\left( Z\right) L\left( Z\right) $ where $c_{1}\left( Z\right) $ is
the lower bound constant from (\ref{optimal3}), i.e. as defined in (\ref{c1}%
). Finally, we relate (\ref{27LB}) to the Wasserstein distance through the
following lemma, proved in the appendix.

\begin{lemma}
\label{LBdWlemma}Lower bound statements in Proposition \ref{CLT for
quadratic var} hold for $d_{W}$ with an additional factor $2$.
\end{lemma}

Thus, under Condition (\ref{BM}), by the previous development and Lemma \ref%
{LBdWlemma}, with%
\begin{equation}
c_{2}\left( Z\right) :=2c_{1}\left( Z\right) L\left( Z\right) ,  \label{c2}
\end{equation}%
we finally get%
\begin{equation}
d_{W}\left( F_{f_{2},n}(Z),N\right) \geq \frac{c_{2}\left( Z\right) }{\sqrt{n%
}},  \label{dWc2}
\end{equation}%
and we are ready to state and prove our lower bound theorem in the quadratic
case under a sharpening of condition (\ref{hypothesis on Z+Y}) and a
quantitative version of Condition (\ref{BM}).

\begin{theorem}
\label{OptimalTheorem}Assume the following two conditions.

\begin{itemize}
\item Let $\left( Y_{k}\right) _{k\in \mathbb{Z}}$ be a process such that
for all $n\in \mathbb{N}$, for some finite constant $c_{3}>0$,%
\begin{equation}
\left\Vert Q_{f_{2},n}(Z+Y)-Q_{f_{2},n}(Z)\right\Vert _{L^{1}(\Omega
)}\leqslant \frac{c_{3}\sqrt{u_{f_{2}}\left( Z\right) }}{n}.  \label{23NEW}
\end{equation}

\item Condition (\ref{BM}) holds and for some finite constant $c_{4}>0$%
\begin{equation}
\left\vert u_{f_{2}}\left( Z\right) -E\left[ U_{f_{2},n}^{2}(Z)\right]
\right\vert \leqslant \frac{2~c_{4}~u_{f_{2}}\left( Z\right) }{\sqrt{n}}.
\label{youyou}
\end{equation}
\end{itemize}

With the positive constants $c_{1}\left( Z\right) ,$ $C_{1}\left( Z\right) $%
, and $L\left( Z\right) $ defined via (\ref{c1}), (\ref{Ldef}), and $%
c_{2}=2c_{1}\left( Z\right) L\left( Z\right) $ (\ref{c2}), which exist by
Proposition \ref{CLT for quadratic var}, if $c_{4}<c_{2}-c_{3}$, then for
any $\varepsilon >0$ such that $c_{2}-\left( 1+\varepsilon \right) \left(
c_{3}+c_{4}\right) >0$, there exists $n_{0}$ large enough that for all $%
n>n_{0}$,%
\begin{equation*}
\frac{c_{1}\left( Z\right) L\left( Z\right) -\left( 1+\varepsilon \right)
\left( c_{3}+c_{4}\right) }{\sqrt{n}}\leqslant d_{W}\left( \frac{%
U_{f_{2},n}\left( Y+Z\right) }{\sqrt{u_{f_{2}}\left( Z\right) }},N\right)
\leqslant \frac{C_{1}\left( Z\right) L\left( Z\right) +c_{3}+c_{4}}{\sqrt{n}}%
.
\end{equation*}
\end{theorem}

\begin{proof}
By using Lemma \ref{dWL1}, the lower bound (\ref{dWc2}) implies%
\begin{equation*}
\frac{c_{2}}{\sqrt{n}}\leqslant d_{W}\left( \frac{U_{f_{2},n}\left(
Y+Z\right) }{\sqrt{E\left[ U_{f_{2},n}\left( Z\right) ^{2}\right] }}%
,N\right) +\frac{1}{\sqrt{E\left[ U_{f_{2},n}\left( Z\right) ^{2}\right] }}E%
\left[ \left\vert U_{f_{2},n}\left( Y+Z\right) -U_{f_{2},n}\left( Z\right)
\right\vert \right] .
\end{equation*}%
Then by assumption (\ref{23NEW}),%
\begin{equation*}
\frac{c_{2}}{\sqrt{n}}\leqslant d_{W}\left( \frac{U_{f_{2},n}\left(
Y+Z\right) }{\sqrt{E\left[ U_{f_{2},n}\left( Z\right) ^{2}\right] }}%
,N\right) +\frac{\sqrt{u_{f_{2}}\left( Z\right) }}{\sqrt{E\left[
U_{f_{2},n}\left( Z\right) ^{2}\right] }}\frac{c_{3}}{\sqrt{n}}.
\end{equation*}%
Now using the trivial consequence of Lemma \ref{dWL1} by which, for any
random variable $Z$ and constants $a,b$, $d_{W}\left( aZ,N\right) \leqslant
d_{W}\left( bZ,N\right) +\left\vert a-b\right\vert \left\Vert Z\right\Vert
_{L^{1}\left( \Omega \right) }$, we get%
\begin{eqnarray}
\frac{c_{2}}{\sqrt{n}} &\leqslant &d_{W}\left( \frac{U_{f_{2},n}\left(
Y+Z\right) }{\sqrt{u_{f_{2}}\left( Z\right) }},N\right)  \notag \\
&&+\left\vert \frac{1}{\sqrt{u_{f_{2}}\left( Z\right) }}-\frac{1}{\sqrt{E%
\left[ U_{f_{2},n}\left( Z\right) ^{2}\right] }}\right\vert E\left[
\left\vert U_{f_{2},n}\left( Z+Y\right) \right\vert \right] +\frac{\sqrt{%
u_{f_{2}}\left( Z\right) }}{\sqrt{E\left[ U_{f_{2},n}\left( Z\right) ^{2}%
\right] }}\frac{c_{3}}{\sqrt{n}}.  \label{c3claim}
\end{eqnarray}%
Regarding the middle term in the right-hand side above, we claim the
following: for any $\varepsilon >0$ and for $n$ large enough,%
\begin{equation}
\left\vert \frac{1}{\sqrt{u_{f_{2}}\left( Z\right) }}-\frac{1}{\sqrt{E\left[
U_{f_{2},n}\left( Z\right) ^{2}\right] }}\right\vert E\left[ \left\vert
U_{f_{2},n}\left( Z+Y\right) \right\vert \right] \leqslant \frac{c_{4}\left(
1+\varepsilon \right) }{\sqrt{n}}.  \label{c4claim}
\end{equation}%
Let us prove this claim. To lighten the notation, we drop the subscripts. By
assumption (\ref{23NEW}), we have%
\begin{eqnarray*}
&&\left\vert \frac{1}{\sqrt{u\left( Z\right) }}-\frac{1}{\sqrt{E\left[
U\left( Z\right) ^{2}\right] }}\right\vert E\left[ \left\vert U\left(
Z+Y\right) \right\vert \right] \\
&\leqslant &\left\vert \frac{\sqrt{n}}{\sqrt{u\left( Z\right) }}-\frac{\sqrt{%
n}}{\sqrt{E\left[ U\left( Z\right) ^{2}\right] }}\right\vert \left( E\left[
\left\vert Q\left( Z\right) -\lambda \left( Z\right) \right\vert \right] +E%
\left[ \left\vert Q\left( Z+Y\right) -Q\left( Z\right) \right\vert \right]
\right) \\
&\leqslant &\left\vert \frac{1}{\sqrt{u\left( Z\right) }}-\frac{1}{\sqrt{E%
\left[ U\left( Z\right) ^{2}\right] }}\right\vert \left( \sqrt{E\left[
U\left( Z\right) ^{2}\right] }+\frac{c_{3}\sqrt{u\left( Z\right) }}{\sqrt{n}}%
\right) .
\end{eqnarray*}%
Since $E\left[ U\left( Z\right) ^{2}\right] $ converges to $u\left( Z\right)
$, and after some simple algebra, for $n$ large enough we get%
\begin{eqnarray*}
\left\vert \frac{1}{\sqrt{u\left( Z\right) }}-\frac{1}{\sqrt{E\left[ U\left(
Z\right) ^{2}\right] }}\right\vert E\left[ \left\vert U\left( Z+Y\right)
\right\vert \right] &\leqslant &\left\vert \frac{1}{\sqrt{u\left( Z\right) }}%
-\frac{1}{\sqrt{E\left[ U\left( Z\right) ^{2}\right] }}\right\vert \left(
1+\varepsilon \right) \sqrt{u\left( Z\right) } \\
&\leqslant &\frac{1+\varepsilon }{2u\left( Z\right) }\left\vert u\left(
Z\right) -E\left[ U\left( Z\right) ^{2}\right] \right\vert .
\end{eqnarray*}%
Thus (\ref{c4claim}) follows immediately from assumption (\ref{youyou}).
Combining (\ref{c4claim}) with (\ref{c3claim}), and again using the
convergence of $E\left[ U_{f_{2},n}\left( Z\right) ^{2}\right] $ to $%
u_{f_{2}}\left( Z\right) $, we finally obtain that for any $\varepsilon >0$
and for $n$ large enough%
\begin{equation*}
\frac{c_{2}}{\sqrt{n}}\leqslant d_{W}\left( \frac{U_{f_{2},n}\left(
Y+Z\right) }{\sqrt{u_{f_{2}}\left( Z\right) }},N\right) +\frac{\left(
c_{3}+c_{4}\right) \left( 1+\varepsilon \right) }{\sqrt{n}},
\end{equation*}%
Since, $c_{4}+c_{4}<c_{2}$, $\varepsilon >0$ exists such that $c_{2}-\left(
1+\varepsilon \right) \left( c_{3}+c_{4}\right) >0$, which finishes the
lower bound of the theorem.

The upper bound is easier to prove, and follows from the same estimates as
for the lower bound. Details are omitted.
\end{proof}

\begin{remark}
\label{NNCLT} The non-central limit theorem in Remark \ref{NCLT} part (1)
also holds if $Z$ is replaced by $Z+Y$ under assumption (\ref{hypothesis on
Z+Y}) if $\gamma >1/2$ ; and similarly for part (2) if $\gamma >H-\left(
q-1\right) /2q$. These results' proofs, which are omitted, follow the
results in Remark \ref{NCLT} and from the tools in this section and those in
\cite{NV2014} and \cite{BN}.
\end{remark}

\subsection{Towards a Berry-Esséen theorem for parameter estimators\label%
{Towards}}

In the previous two sections, we saw how to prove asymptotically normality
for the empirical sums of the form $U_{f_{q},n}\left( Z\right) $ (or $%
U_{f_{q},n}\left( Y+Z\right) $ where $Y$ is a non-stationary correction
process), with convergence speed theorems in total variation and Wasserstein
distances. These apply to parameter estimation if the quantity one is after
is the expected value \thinspace $\lambda _{f_{q}}(Z):=E\left[ f_{q}\left(
Z_{0}\right) \right] $. In this section we evaluate the same question if the
parameter one seeks is implicit in $\lambda _{f_{q}}(Z)$.

Thus assume that one is looking for the unknown parameter $\theta >0$ and
that there is a homeomorphism $g$ such that
\begin{equation*}
\lambda _{f_{q}}(Z)=g^{-1}(\theta ):=\theta ^{\ast }.
\end{equation*}%
As stated, so far, for a degree-$q$ polynomial $f_{q}$ we have studied the
\textquotedblleft estimator\textquotedblright\
\begin{equation*}
\widehat{\theta }_{n}=Q_{f_{q},n}(Z)=\frac{1}{n}\sum_{i=0}^{n-1}f_{q}(Z_{i}).
\end{equation*}%
We have proved the following in Section \ref{PARAMESTIM} (see for instance
Theorems \ref{consistency for Z} and \ref{CLT}, Corollaries \ref{CLTCor} and %
\ref{CLTCor2}) : $\widehat{\theta }_{n}\longrightarrow \theta ^{\ast }$
almost surely and
\begin{equation*}
d_{W}\left( \frac{\sqrt{n}}{\sqrt{E\left[ U_{f_{q},n}^{2}(Z)\right] }}\left(
\widehat{\theta }_{n}-\theta ^{\ast }\right) ,\mathcal{N}(0,1)\right)
\leqslant \varphi (n)
\end{equation*}%
where
\begin{equation*}
U_{f_{q},n}(Z)=\sqrt{n}\left( Q_{f_{q},n}(Z)-\lambda _{f_{q}}(Z)\right) =%
\widehat{\theta }_{n}-\theta ^{\ast }.
\end{equation*}%
and where $\varphi \left( n\right) $ tends to $0$ as $n\rightarrow \infty $
at various speeds which can be determined thanks to the precise statements
in Corollary \ref{CLTCor}, for instance $\varphi \left( n\right) =1/\sqrt{n}$
in Corollary \ref{CLTCor2}, which is the classical Berry-Esséen speed.

By using the relation between $\theta $ and $\lambda $, we naturally define
the estimator of $\theta $ by
\begin{equation*}
\check{\theta}_{n}:=g\left( \widehat{\theta }_{n}\right) .
\end{equation*}%
This is a consistent estimator by Theorem \ref{consistency for Z} since $g$
is continuous by assumption: $\check{\theta}_{n}\longrightarrow \theta $
almost surely. Now assume $g$ is a diffeomorphism. By the mean-value theorem
we can write
\begin{equation*}
\left( \check{\theta}_{n}-\theta \right) =g^{\prime }(\xi _{n})\left(
\widehat{\theta }_{n}-\theta ^{\ast }\right)
\end{equation*}%
where $\xi _{n}$ is a random variable which belongs to $[|\widehat{\theta }%
_{n},\theta ^{\ast }|]$. As a consequence
\begin{eqnarray*}
&&d_{W}\left( \frac{\sqrt{n}}{g^{\prime }(\theta ^{\ast })\sqrt{E\left[
U_{f_{q},n}^{2}(Z)\right] }}\left( \check{\theta}_{n}-\theta \right) ,%
\mathcal{N}(0,1)\right) \\
&\leqslant &d_{W}\left( \frac{\sqrt{n}}{g^{\prime }(\theta ^{\ast })\sqrt{E%
\left[ U_{f_{q},n}^{2}(Z)\right] }}\left( \check{\theta}_{n}-\theta \right) ,%
\frac{\sqrt{n}}{\sqrt{E\left[ U_{f_{q},n}^{2}(Z)\right] }}\left( \widehat{%
\theta }_{n}-\theta ^{\ast }\right) \right) \\
&&+d_{W}\left( \frac{\sqrt{n}}{\sqrt{E\left[ U_{f_{q},n}^{2}(Z)\right] }}%
\left( \widehat{\theta }_{n}-\theta ^{\ast }\right) ,\mathcal{N}(0,1)\right)
\end{eqnarray*}%
The last term above is controlled by $\varphi \left( n\right) $ as
mentioned. Now assume that $g$ is twice continuously differentiable. Then by
the mean-value theorem again, for $\zeta _{n}$ some random variable which
belongs to $[|\xi _{n},\theta ^{\ast }|]\subset \lbrack |\widehat{\theta }%
_{n},\theta ^{\ast }|]$, the other term above is controlled as%
\begin{eqnarray*}
&&d_{W}\left( \frac{\sqrt{n}}{g^{\prime }(\theta ^{\ast })\sqrt{E\left[
U_{f_{q},n}^{2}(Z)\right] }}\left( \check{\theta}_{n}-\theta \right) ,\frac{%
\sqrt{n}}{\sqrt{E\left[ U_{f_{q},n}^{2}(Z)\right] }}\left( \widehat{\theta }%
_{n}-\theta ^{\ast }\right) \right) \\
&\leqslant &\frac{1}{g^{\prime }(\theta ^{\ast })\sqrt{E\left[
U_{f_{q},n}^{2}(Z)\right] }}E\left\vert \sqrt{n}\left( \widehat{\theta }%
_{n}-\theta ^{\ast }\right) \left( g^{\prime }(\xi _{n})-g^{\prime }(\theta
^{\ast })\right) \right\vert \\
&=&\frac{1}{g^{\prime }(\theta ^{\ast })\sqrt{E\left[ U_{f_{q},n}^{2}(Z)%
\right] }}E\left\vert \sqrt{n}\left( \widehat{\theta }_{n}-\theta ^{\ast
}\right) g^{\prime \prime }(\zeta _{n})\left( \xi _{n}-\theta ^{\ast
}\right) \right\vert
\end{eqnarray*}%
\begin{eqnarray*}
&\leqslant &\frac{1}{g^{\prime }(\theta ^{\ast })\sqrt{E\left[
U_{f_{q},n}^{2}(Z)\right] }}E\left\vert \sqrt{n}\left( \widehat{\theta }%
_{n}-\theta ^{\ast }\right) ^{2}g^{\prime \prime }(\zeta _{n})\right\vert \\
&\leqslant &\frac{1}{g^{\prime }(\theta ^{\ast })\sqrt{E\left[
U_{f_{q},n}^{2}(Z)\right] }}\sqrt{n}\left[ E\left( \left( \widehat{\theta }%
_{n}-\theta ^{\ast }\right) ^{2p}\right) \right] ^{1/p}\left[ E\left(
g^{\prime \prime }(\zeta _{n})^{p^{\prime }}\right) \right] ^{1/p^{\prime }}
\end{eqnarray*}

where $p$ and $p^{\prime }$ are conjugate reals greater than $1$, i.e. $%
1/p+1/p^{\prime }=1$. Moreover
\begin{equation*}
\sqrt{n}\left[ E\left( \left( \widehat{\theta }_{n}-\theta ^{\ast }\right)
^{2p}\right) \right] ^{1/p}\leqslant c_{p}\sqrt{n}E\left( \left( \widehat{%
\theta }_{n}-\theta ^{\ast }\right) ^{2}\right) =O(\frac{1}{\sqrt{n}}).
\end{equation*}%
Therefore the only question left to transfer the quantitative results of
Section \ref{PARAMESTIM} to $\check{\theta}_{n}$ is whether one can prove,
for instance, that $g^{\prime \prime }(\zeta _{n})$ has a bounded moment of
order greater than $1$. We will see several examples in Section \ref{AOUP}
where this is easy to check. More generally, we advocate checking this on a
case-by-case basis when the function $g$ can be identified. In the meantime,
we summarize this discussion with the following general principle, which
follows from the above discussion.

\begin{theorem}
\label{Berryesseentheta}Consider the setup from Corollary \ref{CLTCor}, in
which $\widehat{\theta }_{n}=Q_{f_{q},n}(Z):=\frac{1}{n}%
\sum_{i=0}^{n-1}f_{q}(Z_{i})$ and $\theta ^{\ast }=E\left[ f_{q}\left(
Z_{0}\right) \right] ,$ with $\varphi \left( n\right) $ an upper bound for
the expression in (\ref{CorPoint2b}) which converges to $0$. Assume that
there exists a twice-differentiable invertible function $g$ and a value $%
\theta $ such that
\begin{equation*}
g^{-1}(\theta ):=\theta ^{\ast }.
\end{equation*}%
If $g^{\prime \prime }\left( \widehat{\theta }_{n}\right) $ has a moment of
order greater than $1$ which is bounded in $n$, the expression%
\begin{equation*}
\check{\theta}_{n}:=g\left( \widehat{\theta }_{n}\right)
\end{equation*}%
is a strongly consistent and asymptotically normal estimator of $\theta $
and
\begin{equation*}
d_{W}\left( \frac{\sqrt{n}}{g^{\prime }(\theta ^{\ast })\sqrt{E\left[
U_{f_{q},n}^{2}(Z)\right] }}\left( \check{\theta}_{n}-\theta \right) ,%
\mathcal{N}(0,1)\right) \leqslant C\frac{1}{\sqrt{n}}+\varphi \left( n\right)
\end{equation*}%
where $\varphi \left( n\right) $ is the speed of convergence in Corollary %
\ref{CLTCor}.
\end{theorem}

\section{Improving the rate convergence\label{IRC}}

Consider our usual centered stationary Gaussian sequence $Z$ with
autocorrelation function $r_{Z}$, and define :
\begin{equation*}
Z_{k}^{(0)}=Z_{k},\quad k\in \mathbb{Z}
\end{equation*}%
and for every $p\geq 1$
\begin{equation*}
Z_{k}^{(p)}=Z_{k+1}^{(p-1)}-Z_{k}^{(p-1)},k\in \mathbb{Z}.
\end{equation*}%
It is well known, and easily verified, that if $k\mapsto r_{Z}\left(
k\right) $ decays like $k^{-\alpha }$ for $\alpha >0$ as $k\rightarrow
\infty $, then the $p$th-order finite difference process $Z^{\left( p\right)
}$ defined above is also a centered and stationary Gaussian sequence, with
an autocorrelation function which decays like $k^{-\alpha -2p}$. On the
other hand, the reader will easily check, or consult the computations in
Section 6.5 of \cite{BBNP}, which are also summarized in Proposition 4.2
part (3) in \cite{NP2013}, and extended in \cite{NV2014} to cover the case
at hand here, that as soon as $-\alpha -2p<-3/4$, we have
\begin{equation*}
\kappa _{4}(U_{f_{2},n}(Z^{\left( p\right) }))\leqslant c\left( Z\right)
k^{-1}
\end{equation*}%
as $k\rightarrow \infty $ for some constant $c\left( Z\right) $ depending
only on the law of $Z$. The condition $-\alpha -2p<-3/4$ is satisfied for
any $\alpha >0$, as soon as the integer $p\geq 1$, evidently. For
non-integer $p>0$, we may define fractional finite-differences $Z^{\left(
p\right) }$ using the standard formal power series expansion, in which case
the condition on $p$ becomes
\begin{equation}
p>3/8-\alpha .  \label{pee}
\end{equation}%
Thus by applying Theorem \ref{CLT} and part (3) of Corollary \ref{CLTCor},
we immediately conclude the following, which we state for integer $p$ since
the case of non-integer $p$ is arguably of lesser practical use.

\begin{theorem}
\label{ThmImprove}Let $q$ be an even positive integer. Assume that $Z,$ $%
r_{Z}$, $f_{q}$, $U_{f_{q},n}(Z)$ and $F_{f_{q},n}(Z)$ are as in Theorem \ref%
{CLT}. Assume that $k^{\alpha }r_{Z}(k)$ converges to a constant for some $%
\alpha >0$ as $k\rightarrow \infty $. Then for every integer $p\geq 1$,
there exists $C$ depending on $p,q$ and $r_{Z^{\left( p\right) }}\left(
0\right) $ such that
\begin{equation*}
d_{TV}\left( F_{f_{q,n}}((Z^{(p)})),N\right) \leqslant Cn^{-1/4}.
\end{equation*}%
In particular, $Z^{\left( p\right) }$ satisfies condition (\ref{BM}) and $%
U_{f_{q,n}}((Z^{(p)}))$ is asymptotically normal.

For instance, if $Z$ has autocorrelation asymptotics which are equivalent to
those of a fractional Gaussian noise with Hurst parameter $H$, then $\alpha
=2-2H>0$ and the above statements hold for all $H\in \left( 0,1\right) $.
\end{theorem}

To apply the remainder of Corollary \ref{CLTCor}, we note that, by the
calculation in Section \ref{EXGEN}, under the assumption that $r_{Z}(k)\sim
ck^{-\alpha }$, since $Z^{\left( p\right) }$ satisfies condition (\ref{BM}),
the control on the variance discrepancy term given in (\ref{variancespeed})
is of the same order as:

\begin{equation*}
\sum_{\left\vert j\right\vert >n}r_{Z^{\left( p\right) }}(j)^{2}\sim
\sum_{\left\vert j\right\vert >n}j^{-2\alpha -4p}\asymp n^{-2\alpha
-4p+1}=o\left( n^{-3}\right) .
\end{equation*}%
This is negligible compared to $n^{-1/4}$. Thus we immediately get the
following by part (2) of Corollary \ref{CLTCor}.

\begin{corollary}
\label{CorImprove}Under the notation and conditions of Theorem \ref%
{ThmImprove}, for some constant $c_{p,q}(Z)$ depending only on $p,q,$ and
the law of $Z$,%
\begin{equation*}
d_{TV}\left( U_{f_{q},n}(Z^{\left( p\right) }),\mathcal{N}\left(
0,u_{f_{q}}(Z^{\left( p\right) })\right) \right) \leqslant
c_{p,q}(Z)n^{-1/4}.
\end{equation*}%
This holds for instance if $Z$ is as in the last statement in Theorem \ref%
{ThmImprove}.
\end{corollary}

For the same reason as in the case of $p=0$, there is no reason to believe
that the speed $n^{-1/4}$ is sharp, but when $q=2$, a sharp result can be
established. The results of Section \ref{QC} imply the next sharp Berry-Essé%
en-type theorem.

\begin{theorem}
\label{OptimalImprove}Under all the assumptions of Theorem \ref{ThmImprove},%
\begin{equation*}
d_{W}\left( U_{f_{2},n}(Z^{\left( p\right) }),\mathcal{N}\left(
0,u_{f_{2}}(Z^{\left( p\right) })\right) \right) \asymp \frac{1}{\sqrt{n}}.
\end{equation*}
\end{theorem}

\begin{proof}
We briefly sketch the ideas in the proof. The details are bookkeeping given
what has already been used to prove Corollary \ref{CLTCor} and Theorem \ref%
{OptimalTheorem}. The statement in Theorem \ref{OptimalImprove} holds with $%
U $ replaced by the normalized version $F$, by directly checking all the
assumptions of Theorem \ref{OptimalTheorem}; to get the result for $U$
instead of $F$, an argument of the same type as in the proof of Corollary %
\ref{CLTCor} suffices, which works just as we saw in the above justification
of Corollary \ref{CorImprove}, because one easily shows that the variance
discrepancy term (\ref{variancespeed}) is still $o\left( n^{-3}\right) $.
\end{proof}

\begin{remark}
Improving rates of convergence by using finite differences also works for
the non-stationary processes of Section (\ref{NONSTAT}). The reason is
simply that, since the process $Y$ is a small perturbation of $Y+Z$ in $%
L^{1}\left( \Omega \right) $ under assumptions such as (\ref{hypothesis on
Z+Y}), these assumptions do not deteriorate under a finite difference of
fixed order $p$, up to the possible inclusion of multiplicative constants
depending only on $p,q$. All details are omitted for conciseness.
\end{remark}

\section{Applications to Ornstein-Uhlenbeck processes : the one-parameter
case \label{AOUP}}

\subsection{Fractional Ornstein-Uhlenbeck process: general case\label{FOUgen}%
}

Consider an Ornstein-Uhlenbeck process $X=\left\{ X_{t},t\geq 0\right\} $
driven by a fractional Brownian motion $B^{H}=\left\{ B_{t}^{H},t\geq
0\right\} $ of Hurst index $H\in (0,1)$. That is, $X$ is the solution of the
following linear stochastic differential equation
\begin{equation}
X_{0}=0;\quad dX_{t}=-\theta X_{t}dt+dB_{t}^{H},\quad t\geq 0,  \label{FOU}
\end{equation}%
whereas $\theta >0$ is considered as unknown parameter. The solution $X$ of (%
\ref{FOU}) has the following explicit expression:
\begin{equation}
X_{t}=\int_{0}^{t}e^{-\theta (t-s)}dB_{s}^{H}.  \label{fOUX}
\end{equation}%
Thus, we can write
\begin{equation}
X_{t}=Z_{t}^{\theta }-e^{-\theta t}Z_{0}^{\theta }  \label{decompoXfOU}
\end{equation}%
where
\begin{equation}
Z_{t}^{\theta }=\int_{-\infty }^{t}e^{-\theta (t-s)}dB_{s}^{H}.
\label{Ztheta}
\end{equation}%
Moreover, it is known that $Z^{\theta }$ is an ergodic stationary Gaussian
process. It is the stationary solution of equation (\ref{FOU}). We are thus
in the setup of Section \ref{NONSTAT} with $Z=Z^{\theta }$ and $%
Y=-e^{-\theta t}Z_{0}^{\theta }$. Consequently, to apply the results of that
section, we need only check that Condition \ref{hypothesis on Z+Y} holds. It
does, according to the following result.

\begin{lemma}
\label{asymptotic R_{Q,q}} Let $X$ and $Z^{\theta }$ be the processes given
in (\ref{FOU}) and (\ref{Ztheta}) respectively. Then for every $p\geq 1$ and
for all $n\in \mathbb{N}$,
\begin{equation*}
\left\Vert Q_{f_{q},n}(X)-Q_{f_{q},n}(Z^{\theta })\right\Vert _{L^{p}(\Omega
)}=\mathcal{O}\left( n^{-1}\right) .
\end{equation*}
\end{lemma}

\begin{proof}
By (\ref{decompoXfOU}) and (\ref{polynomial function}) we have
\begin{equation*}
\left\Vert Q_{f_{q},n}(X)-Q_{f_{q},n}(Z^{\theta })\right\Vert _{L^{p}(\Omega
)}\leqslant \frac{1}{n}\sum_{i=0}^{n-1}\sum_{k=0}^{q/2}d_{f_{q},2k}\left%
\Vert H_{2k}\left( \frac{X_{i}}{\sqrt{r_{Z}(0)}}\right) -H_{2k}\left( \frac{%
Z_{i}^{\theta }}{\sqrt{r_{Z}(0)}}\right) \right\Vert _{L^{p}(\Omega )}.
\end{equation*}%
Combining this and the fact that
\begin{equation*}
H_{2k}\left( \frac{X_{i}}{\sqrt{r_{Z}(0)}}\right) -H_{2k}\left( \frac{%
Z_{i}^{\theta }}{\sqrt{r_{Z}(0)}}\right) =\sum_{l=0}^{k}\frac{(2k)!(-1)^{l}}{%
l!(2k-2l)!2^{l}}\sum_{j=1}^{2k-2l}\frac{(-1)^{j}(_{j}^{2k-2l})e^{-\theta ij}%
}{r_{Z}^{k-l}(0)}(Z_{0}^{\theta })^{j}(Z_{i}^{\theta })^{2k-2l-j}.
\end{equation*}%
we deduce that there exist a constant $c(\theta ,f_{q})$ depending on $f_{q}$
and $\theta $ such that
\begin{equation*}
\left\Vert Q_{f_{q},n}(X)-Q_{f_{q},n}(Z^{\theta })\right\Vert _{L^{p}(\Omega
)}\leqslant c(\theta ,f_{q})\frac{1}{n}\sum_{i=0}^{n-1}e^{-i\theta }.
\end{equation*}%
Thus the lemma is obtained.
\end{proof}

As a consequence, by using $Z^{\theta }$ ergodic, Lemma \ref{asymptotic
R_{Q,q}} and Theorem \ref{consistency for Z+Y}, we conclude that
\begin{equation*}
Q_{f_{q},n}(X)\longrightarrow \lambda _{f_{q}}(Z^{\theta })
\end{equation*}%
almost surely as $n\rightarrow \infty $. Moreover, by the Gaussian property
of $Z^{\theta }$ and Lemma asymptotic \ref{hypotheses FOU} in the Appendix,
we can write
\begin{equation*}
\lambda _{f_{q}}(Z^{\theta }):=\mu _{f_{q}}(\theta )
\end{equation*}%
where $\mu _{f_{q}}$ is a univariate function of $\theta $ determined by the
polynomial $f_{q}$. Hence, in the case when the function $\mu _{f_{q}}$ is
invertible, we obtain the following estimator for $\theta $
\begin{equation}
\check{\theta}_{f_{q},n}:=\mu _{f_{q}}^{-1}\left[ Q_{f_{q},n}(X)\right] .
\label{estimator of FOU}
\end{equation}

\begin{proposition}
Assume $H\in \left( 0,1\right) $ and $\mu _{f_{q}}$ is a homomorphism. Let $%
\widehat{\theta }_{f_{q},n}$ be the estimator given in (\ref{estimator of
FOU}). Then, as $n\longrightarrow \infty $, almost surely, $\check{\theta}%
_{f_{q},n}\longrightarrow \theta $.
\end{proposition}

These considerations allow us to state and prove the following strong
consistency and asymptotic normality of $\check{\theta}_{f_{q},n}$ .

\begin{proposition}
\label{asymptotic distribution of Q_{q,n}(FOU)} Denote $N\sim \mathcal{N}%
(0,1)$. Then

\begin{itemize}
\item If $H\in (0,\frac{5}{8})$, for any $q$,%
\begin{equation*}
d_{W}\left( \frac{U_{f_{q},n}(X)}{\sqrt{u_{f_{q}}(Z^{\theta })}},N\right)
\leqslant \frac{C}{n^{1/4}}.
\end{equation*}

\item If $H\in (\frac{5}{8},\frac{3}{4})$, for any $q$,%
\begin{equation*}
d_{W}\left( \frac{U_{f_{q},n}(X)}{\sqrt{u_{f_{q}}(Z^{\theta })}},N\right)
\leqslant \frac{C}{n^{(4H-3)/2}}.
\end{equation*}

\item In particular, in both cases, assuming $\mu _{f_{q}}$ is a
diffeomorphism,%
\begin{equation*}
\sqrt{n}\left( \check{\theta}_{f_{q},n}-\theta \right) \overset{law}{%
\longrightarrow }\mathcal{N}\left( 0,\frac{u_{f_{q}}(Z^{\theta })}{(\mu
_{f_{q}}^{\prime }(\theta ))^{2}}\right) .
\end{equation*}

\item If $H=\frac{3}{4}$,
\begin{equation*}
d_{W}\left( \frac{U_{f_{q},n}(X)}{\sqrt{E\left[ U_{f_{q},n}^{2}(Z^{\theta })%
\right] }},N\right) \leqslant C\log ^{-\frac{1}{4}}(n).
\end{equation*}%
In particular,
\begin{equation*}
\sqrt{\frac{n}{\log (n)}}\left( \check{\theta}_{f_{q},n}-\theta \right)
\overset{law}{\longrightarrow }\mathcal{N}\left( 0,\frac{9d_{q,2}^{2}(Z)}{%
16\theta ^{4}(\mu _{f_{q}}^{\prime }(\theta ))^{2}}\right) .
\end{equation*}
\end{itemize}
\end{proposition}

\begin{proof}
In this proof, $C$ represents a constant which may change from line to line.
It was proved in \cite{NV2014} (also see \cite{BBNP}) that, with $r_{Z}$ the
covariance function of $Z^{\theta }$,
\begin{equation*}
\kappa _{4}(U_{f_{2},n}(Z^{\theta }))\leqslant C\left( \sum_{\left\vert
k\right\vert <n}\left\vert r_{Z}\left( k\right) \right\vert ^{4/3}\right)
^{3}
\end{equation*}%
while
\begin{equation*}
E\left[ U_{f_{q},n}^{2}(Z^{\theta })\right] \leqslant C\sum_{\left\vert
k\right\vert <n}\left\vert r_{Z}\left( k\right) \right\vert ^{2}
\end{equation*}%
Then by Lemma \ref{hypotheses FOU}, for all $H<3/4$, we easily get $E\left[
U_{f_{q},n}^{2}(Z^{\theta })\right] \leqslant u_{f_{q}}(Z^{\theta })<\infty $
and in particular%
\begin{equation}
\left\vert 1-\frac{E\left[ U_{f_{q},n}^{2}(Z^{\theta })\right] }{%
u_{f_{q}}(Z^{\theta })}\right\vert \leqslant C\sum_{\left\vert k\right\vert
>n}\left\vert r_{Z}\left( k\right) \right\vert ^{2}\leqslant Cn^{4H-3}.
\label{u-U for fOU}
\end{equation}%
Also by Lemma \ref{hypotheses FOU}, for $H<5/8$,%
\begin{equation*}
\kappa _{4}(U_{f_{2},n}(Z^{\theta }))\leqslant Cn^{-1}
\end{equation*}%
while for $H>5/8$,%
\begin{equation*}
\kappa _{4}(U_{f_{2},n}(Z^{\theta }))\leqslant Cn^{2\left( 4H-3\right) }.
\end{equation*}%
Then by Theorem \ref{CLT for Z+Y} and by Lemma \ref{asymptotic R_{Q,q}}
which shows that $\gamma =1$, we get%
\begin{equation*}
d_{W}\left( \frac{U_{f_{q},n}(Z+Y)}{\sqrt{u_{f_{q}}(Z)}},N\right) \leqslant
Cn^{-1/2}+\sqrt[4]{\kappa _{4}(U_{f_{2},n}(Z^{\theta }))}+Cn^{4H-3}
\end{equation*}%
depending on whether $H$ is larger or smaller than $5/8$ we get the
announced result, since $n^{-1/2}$ and $n^{\left( 4H-3\right) /2}$ coming
from dominate the term $\sqrt[4]{\kappa _{4}(U_{f_{2},n}(Z^{\theta }))}$
dominate the error terms $n^{4H-3}$and $n^{-1/2}$.

Now, by assumption, $\mu _{f_{q}}$ has a continuously differentiable
derivative. Thus, by the mean value theorem, there exists a random variable $%
\xi _{f_{q},n}$ between $\theta $ and $\widehat{\theta }_{f_{q},n}$ such
that
\begin{equation*}
\sqrt{n}\left( \mu _{f_{q}}(\widehat{\theta }_{q,n})-\mu _{f_{q}}(\theta
)\right) =\mu _{f_{q}}^{\prime }(\xi _{f_{q},n})\sqrt{n}\left( \widehat{%
\theta }_{q,n}-\theta \right) .
\end{equation*}%
By the normal convergence in law of $\mu _{f_{q}}(\widehat{\theta }_{q,n})$
to $\mu _{f_{q}}(\theta )$ and the almost-sure convergence of $\mu
_{f_{q}}^{\prime }(\xi _{f_{q},n})$ to $\mu _{f_{q}}^{\prime }\left( \theta
\right) $, the theorem's final statement when $H<3/4$ follows. The special
case of $H=3/4$ is treated similarly.
\end{proof}

\subsection{Examples and a Berry-Esséen theorem for drift estimators\label%
{FOUEx}}

In the two following examples, the function $\mu _{f_{q}}$ is an explicit
diffeomorphism except at $\theta =0$.

\begin{itemize}
\item Assume that $f_{q}=H_{q}$. Using (\ref{alphaHqcase}) and Lemma \ref%
{hypotheses FOU}, we have
\begin{equation*}
\mu _{H_{q}}(\theta )=\lambda _{H_{q}}(Z^{\theta })=\frac{q!}{(\frac{q}{2}%
)!2^{q/2}}\left( H\Gamma (2H)\theta ^{-2H}-1\right) ^{q/2}.
\end{equation*}%
In this case, the function $\mu $ is a diffeomorphism with bounded
derivatives when the range is restricted to $\mathbf{R}_{+}$. Since, by the
previous strong consistency proposition, $Q_{f_{q},n}(X)$ ends up in $%
\mathbf{R}_{+}$ almost surely, the estimator $\check{\theta}_{f_{q},n}$ is
asymptotically equivalent to the one in which the function $g=\mu
_{f_{q}}^{-1}$ is restricted to $\mathbf{R}_{+}$. This observation will be
helpful below when applying the results of Section \ref{Towards}.

\item Assume that $f_{q}=\phi _{q}$ with $\phi _{q}(x)=x^{q}$. From (\ref%
{alphaPhicase}) and Lemma \ref{hypotheses FOU} we obtain
\begin{equation*}
\mu _{\phi _{q}}(\theta )=\lambda _{\phi _{q}}(Z^{\theta })=\frac{q!}{(\frac{%
q}{2})!2^{q/2}}\left[ H\Gamma (2H)\theta ^{-2H}\right] ^{q/2}.
\end{equation*}%
The singularity of $\mu $ at $\theta =0$ poses some technical problems when
one tries to translate the consistency of $Q_{f_{q},n}(X)$ into that of $%
\check{\theta}_{f_{q},n}$ thanks to Section \ref{Towards}, which we
investigate below.
\end{itemize}

We now show how the principle described in Section \ref{Towards} can be used
to estimate the speed of convergence for the estimator $\check{\theta}%
_{f_{q},n}$ itself. To work in a specific situation, we look at the above
two examples, assuming $q=2$.

\subsubsection{Berry-Esséen theorem for a Hermite-variations-based estimator
for $\protect\theta $}

In the notation of Section \ref{Towards}, using the convention of replacing $%
Q_{H_{2},n}\left( Z\right) $ by $\left\vert Q_{H_{2},n}\left( Z\right)
\right\vert $, in the case of the Hermite polynomial $H_{2}$ we have%
\begin{equation*}
\mu _{H_{2}}(\theta )=\lambda _{H_{2}}(Z^{\theta })=H\Gamma (2H)\left\vert
\theta \right\vert ^{-2H}-1=g^{-1}\left( \theta \right)
\end{equation*}%
and thus
\begin{equation}
g\left( x\right) =g_{H_{2}}\left( x\right) :=\left( H\Gamma (2H)\right)
^{-1/(2H)}\left( 1+\left\vert x\right\vert \right) ^{-1/(2H)},  \label{gee}
\end{equation}%
and $g^{\prime \prime }\left( x\right) $ is proportional to $\left(
1+\left\vert x\right\vert \right) ^{-1/(2H)-2}$. This function is bounded on
$\mathbf{R}_{+}$. Hence, according to Theorem \ref{Berryesseentheta}, using
the speed of convergence from Proposition \ref{CLT for quadratic var}, we
obtain the following.

\begin{proposition}
\label{BerryHermiteThetaProp}For the stationary fractional
Ornstein-Uhlenbeck $Z^{\theta }$ in (\ref{Ztheta}), with $Q_{H_{2},n}\left(
Z^{\theta }\right) =\frac{1}{n}\sum_{k=1}^{n}Z^{\theta }\left( k\right)
^{2}-1$ and $g$ as in (\ref{gee})$,$ we get
\begin{equation*}
d_{W}\left( \frac{\sqrt{n}}{g^{\prime }(\theta ^{\ast })\sqrt{E\left[
U_{H_{2},n}^{2}(Z)\right] }}\left( g\left( Q_{H_{2},n}\left( Z^{\theta
}\right) \right) -\theta \right) ,\mathcal{N}(0,1)\right) \leqslant C\frac{1%
}{\sqrt{n}}+\varphi \left( n\right)
\end{equation*}%
where $g\left( x\right) =\left( H\Gamma (2H)\right) ^{-1/(2H)}\left(
1+\left\vert x\right\vert \right) ^{-1/(2H)}$ and $\theta ^{\ast
}=g^{-1}\left( \theta \right) $ and%
\begin{equation*}
\varphi \left( n\right) \asymp \frac{\left( \sum_{|k|<n}\left\vert
k\right\vert ^{3H-3}\right) ^{2}}{\left( \sum_{|k|<n}\left\vert k\right\vert
^{4H-4}\right) ^{3/2}\sqrt{n}}.
\end{equation*}%
In particular, for $H<2/3$, $\varphi \left( n\right) \asymp 1/\sqrt{n}$.
\end{proposition}

\begin{remark}
Improvements to the above proposition which include the use of the
asymptotic variance $u_{f_{q}}(Z^{\theta })$ and the nonstationary
fractional OU process $X$ in (\ref{fOUX}) also hold. These are omitted here
for the sake of conciseness; the reader will find these topics covered in
Section \ref{OBEQ} below.
\end{remark}

\subsubsection{Comments and strategy for $\protect\theta $ estimators with
singular variance function}

For the power-2 function, in the notation of Section \ref{Towards} we have%
\begin{equation*}
\mu _{\phi _{2}}(\theta )=\lambda _{\phi _{2}}(Z^{\theta })=H\Gamma
(2H)\theta ^{-2H}=g^{-1}\left( \theta \right)
\end{equation*}%
and thus $g^{\prime \prime }\left( x\right) =c_{H}x^{-1/(2H)-2}$ which has a
singularity at $0$ and thus is not bounded. A result can be obtained
immediately from the above proposition since $g\left( Q_{\phi _{2},n}\left(
Z^{\theta }\right) -1\right) =g\left( Q_{H_{2},n}\left( Z^{\theta }\right)
\right) $ is the estimator studied in that proposition. However, for
illustrative purposes, we finish this section by outlining a method for
dealing with the singularity, since this works for any $g$ such that $%
g^{\prime \prime }$ is asymptotically decreasing like a negative power, and
any process that has an infinite Karhunen-Loève expansion.

According to Theorem \ref{Berryesseentheta}, we are asking whether for some $%
p^{\prime }>1$,
\begin{equation*}
\sup_{n}E\left[ \left\vert \widehat{\theta }_{n}\right\vert ^{-p^{\prime
}/(2H)-2p^{\prime }}\right] <\infty .
\end{equation*}%
This condition is not entirely trivial, and can fail in some simple
pathologically degenerate cases such as if $Z$ is constant, since then $%
\widehat{\theta }_{n}=n^{-1}\sum_{k=1}^{n}Z\left( k\right) ^{2}=Z\left(
0\right) ^{2}$ is a chi-squared variable with one degree of freedom, which
has no moments of negative order less than $-1/2$. This pathology does not
occur for the fOU process, though the argument is slightly involved, since
the limit law of the renormalized $\widehat{\theta }_{n}$ is normal, which
does not have higher negative moments either. We decompose%
\begin{equation*}
E\left[ g^{\prime \prime }\left( \widehat{\theta }_{n}\right) \right] =E%
\left[ g^{\prime \prime }\left( \widehat{\theta }_{n}\right) \mathbf{1}_{%
\widehat{\theta }_{n}<1/\sqrt{n}}\right] +E\left[ g^{\prime \prime }\left(
\widehat{\theta }_{n}\right) \mathbf{1}_{\widehat{\theta }_{n}\geq 1/\sqrt{n}%
}\right] .
\end{equation*}%
By the asymptotic normality of $\left( \widehat{\theta }_{n}-\theta ^{\ast
}\right) \sqrt{n}$, we get%
\begin{eqnarray*}
&&E\left[ g^{\prime \prime }\left( \widehat{\theta }_{n}\right) \mathbf{1}%
_{\left\vert \widehat{\theta }_{n}\right\vert \geq 1/\sqrt{n}}\right] \sim
\int_{-\sqrt{n}\theta ^{\ast }+1}^{\infty }\left\vert \frac{z}{\sqrt{n}}%
+\theta ^{\ast }\right\vert ^{-p^{\prime }\left( 2+1/(2H)\right)
}e^{-z^{2}/2}dz \\
&=&\int_{-\sqrt{n}\theta ^{\ast }+1}^{-\sqrt{n}\theta ^{\ast }/2}\left\vert
\frac{z}{\sqrt{n}}+\theta ^{\ast }\right\vert ^{-p^{\prime }\left(
2+1/(2H)\right) }e^{-z^{2}/2}dz+\int_{-\sqrt{n}\theta ^{\ast }/2}^{\infty
}\left\vert \frac{z}{\sqrt{n}}+\theta ^{\ast }\right\vert ^{-p^{\prime
}\left( 2+1/(2H)\right) }e^{-z^{2}/2}dz \\
&\leqslant &cst~e^{-n\theta ^{\ast }/8}n^{p^{\prime }\left( 2+1/(2H)\right)
}+\left( \theta ^{\ast }/2\right) ^{-p^{\prime }\left( 2+1/(2H)\right) }
\end{eqnarray*}%
which is bounded for all $n$.

For the second piece, the normal approximation would not yield a finite
bound, thus we must return to the original expression of $\widehat{\theta }%
_{n}$ as a 2nd chaos variable. It is known (see \cite[page 522]{LP}) that $%
Z^{\theta }\left( k\right) $ has a Kahunen-Loève expansion $%
\sum_{m=0}^{\infty }\sqrt{\lambda _{m}}e_{m}\left( k\right) W_{m}$ (where
the $W_{m}$ are i.i.d. standard normal, and the $e_{m}$ are orthonormal in $%
L^{2}\left( [0,n]\right) $ ) such that $\lambda _{m}\sim cm^{2H-2}$. Thus,
the expansion of $Z^{\theta }$ contains infinitely many independent terms.
One also knows (see \cite[Section 2.7.4]{NP-book}) that $\hat{\theta}_{n}$,
like any variable in the second chaos, can be expanded as $%
\sum_{m=0}^{\infty }\mu _{m}W_{m}^{2}$ where the $\mu _{m}$ are summable.
One can check that the infinity of distinct terms in the expansion of $Z^{y}$
implies that for any fixed $n$, the expansion of $\hat{\theta}_{n}$ also
contains infinitely many terms, and that the coefficients are positive.
Therefore, for any fixed $m_{0}$, there exists a positive constant $c_{m_{0}}
$ such that $\hat{\theta}_{n}\geq
c_{m_{0}}\sum_{m=0}^{m_{0}}W_{m}^{2}=:S_{m_{0}}$, which is a random variable
with $\chi ^{2}$ distribution with $m_{0}$ degrees of freedom. Hence the
density of $S$ at the origin is of the order $z^{(m_{0}-1)/2}$, which means
it has a negative moment of order $-p^{\prime }\left( 2+1/(2H)\right) $ as
soon as $m_{0}>p^{\prime }\left( 2+1/(2H)\right) -1$. From this it follows
that $E\left[ g^{\prime \prime }\left( \widehat{\theta }_{n}\right) \mathbf{1%
}_{\widehat{\theta }_{n}<1/\sqrt{n}}\right] $ is bounded. Thus $g$ and $\hat{%
\theta}$ comply with the conditions of Theorem \ref{Berryesseentheta}. In
fact, since $p^{\prime }$ can be taken arbitrarily close to $1$, we only
need to be able to choose $m_{0}>1+1/(2H)$. For instance, if $H>1/2$, this
means that for Theorem \ref{Berryesseentheta} to work with $q=2$, one only
needs to show that the second-chaos series decomposition of $\widehat{\theta
}_{n}$ contains $2$ independent terms. This covers all Gaussian processes
except for the trivial case of the constant process.

\subsection{Optimal Berry-Esséen theorem in the quadratic case\label{OBEQ}}

As in the previous sections, the convergence speed for general $q$ has no
reason to be optimal. We illustrate this by studying the case $q=2$, where
we can improve the rate convergence thanks to the optimal rates obtained in
Section \ref{QC}, and even obtain optimal two-sided bounds when $H<5/8$.
Note that the results in this section deal with the fully realistic scenario
where observations come from the non-stationary process $X$ in (\ref{fOUX})
and there is no reference to normalizing constants other than finite
asymptotic variances.

\subsubsection{Setting up the rates of convergence}

First assume that $H\leqslant 3/4$. We find that $\kappa
_{3}(F_{f_{2},n}(Z^{\theta }))\longrightarrow 0$ and more precisely (see
\cite{NV2014}) that%
\begin{equation}
\left\vert E\left( (F_{f_{2},n}(Z^{\theta }))^{3}\right) \right\vert \asymp
\frac{\left( \sum_{|k|<n}|r_{Z}(k)|^{3/2}\right) ^{2}}{\left(
\sum_{|k|<n}|r_{Z}(k)|^{2}\right) ^{3/2}\sqrt{n}}\leqslant C\times \left\{
\begin{aligned} n^{-\frac12},\quad \mbox{if } 0<H<\frac23\\ \log^2(n)n^{-\frac12},\quad \mbox{if
} H=\frac23\\ n^{6H-\frac92},\quad \mbox{if } \frac23<H<\frac34\\ \log^{-3/2}(n),\quad \mbox{if } H=\frac34. \end{aligned}%
\right.  \label{exact rates}
\end{equation}%
By using Corollary \ref{cor CLT for quadratic var}, we see that we must
compare the rates therein to the rates obtained in (\ref{exact rates}). By (%
\ref{u-U for fOU}) the rate which controls the convergence of the variances
is $n^{4H-3}$. This can be dominated by $1/\sqrt{n}$ if and only if $n<5/8$.
For $H\in \lbrack 2/3,3/4)$, $n^{4H-3}$ dominates the rates in (\ref{exact
rates}). The rate which controls the non-stationarity term is always of
order $1/\sqrt{n}$ because of Lemma \ref{asymptotic R_{Q,q}}, which is
always the lowest-order term. Hence the improved rates in (\ref{exact rates}%
) only come into play when $H<5/8$ when normalizing by the asymptotic
variance. In other words, we have the following two estimates, where the
second one avoids the use of non-empirical statistics.

\begin{proposition}
If $H\in (0,\frac{3}{4}]$,
\begin{equation*}
d_{W}\left( \frac{U_{f_{2},n}(X)}{\sqrt{E\left[ U_{f_{2},n}^{2}(Z^{\theta })%
\right] }},N\right) \leqslant C\times \left\{
\begin{aligned} n^{-\frac12},\quad \mbox{if } 0<H<\frac23\\ \log^2(n)n^{-\frac12},\quad \mbox{if
} H=\frac23\\ n^{6H-\frac92},\quad \mbox{if } \frac23<H<\frac34\\ \log^{-3/2}(n),\quad \mbox{if } H=\frac34 \end{aligned}%
\right.
\end{equation*}%
and
\begin{equation*}
d_{W}\left( U_{f_{2},n}(X),\sqrt{u_{f_{2}}(Z^{\theta })}N\right) \leqslant
C\times \left\{ \begin{aligned} n^{-\frac12},\quad \mbox{if } 0<H<\frac58\\
n^{4H-3},\quad \mbox{if } \frac58\leqslant H<\frac34. \end{aligned}\right.
\end{equation*}
\end{proposition}

Next we show how to obtain optimal rates of convergence in the Wasserstein
distance when $H<5/8$. This result is important methodologically speaking
because, thanks to the strategy in Section \ref{IRC}, it is essentially
always possible to transform one's long-memory time series into one which
satisfies $H<5/8$, by using a finite-difference transformation.

\subsubsection{Applying the optimal theorem}

To apply Theorem \ref{OptimalTheorem} we must check that conditions (\ref%
{23NEW}) and (\ref{youyou}) are met. We just saw that this is the case when $%
H<5/8$. However, we must also check that the corresponding constants $c_{3}$
and $c_{4}$ are sufficiently small. Since $H<5/8$, by (\ref{u-U for fOU}),
the constant $c_{4}$ can be made arbitrarily small for $n$ large enough. It
remains to show that $c_{3}$ can be chosen small. A direct application of
Lemma \ref{asymptotic R_{Q,q}} is insufficient for this purpose. Therefore,
we must modify our estimator slightly, by discarding some of the first
terms. We thus fix an integer $i_{0}>0$ and define%
\begin{equation}
\tilde{Q}_{f_{2},n}\left( X\right) :=\frac{1}{n}%
\sum_{i=i_{0}}^{i_{0}+n-1}f_{2}\left( X_{i}\right) .  \label{Qtilda}
\end{equation}%
It is easy to check that this is still a consistent and asymptotically
normal estimator of $r_{Z}\left( 0\right) $. By the proof of Lemma \ref%
{asymptotic R_{Q,q}}, we see that4
\begin{equation}
\left\Vert \tilde{Q}_{f_{2},n}\left( X\right) -\tilde{Q}_{f_{2},n}\left(
Z^{\theta }\right) \right\Vert _{L^{p}\left( \Omega \right) }\leqslant
c\left( \theta ,f_{2}\right) \frac{1}{n}\sum_{i=i_{0}}^{i_{0}+n-1}e^{-i%
\theta }\leqslant c\left( \theta ,f_{2}\right) \frac{1}{n}e^{-i_{0}\theta }.
\label{improved lemma 12}
\end{equation}%
Since $\theta >0$, we can make the last expression above as small as we want
by choosing $i_{0}$ sufficiently large. Thus Theorem \ref{OptimalTheorem}
applies, and we have the following optimal Berry-Esséen theorem for the
variance estimator $\tilde{Q}_{f_{2},n}\left( X\right) $, which, as we saw
in Section \ref{FOUgen}, gives access to estimators for $\theta $.

\begin{proposition}
\label{OptimalBEProp}If $H\in \left( 0,\frac{5}{8}\right) $, then there
exists an integer $i_{0}>0$ such that the quadratic variation $\tilde{Q}%
_{f_{2},n}$ defined in (\ref{Qtilda}) satisfies%
\begin{equation*}
d_{W}\left( \sqrt{n}\left[ \tilde{Q}_{f_{2},n}\left( X\right) -r_{Z}\left(
0\right) \right] ,\mathcal{N}(0,u_{f_{2}}(Z^{\theta }))\right) \asymp \frac{1%
}{\sqrt{n}}.
\end{equation*}%
Moreover, with $g$ as in (\ref{gee}), we have%
\begin{equation*}
d_{W}\left( \sqrt{n}\left( g\left( \tilde{Q}_{H_{2},n}\left( X\right)
\right) -\theta \right) ,\mathcal{N}(0,g^{\prime }\left( \theta ^{\ast
}\right) ^{2}u_{H_{2}}(Z^{\theta }))\right) \asymp \frac{1}{\sqrt{n}}
\end{equation*}
\end{proposition}

\begin{proof}
The first result follows from the considerations immediately above. The
second follows from Theorem \ref{Berryesseentheta} exactly as did the result
in Proposition \ref{BerryHermiteThetaProp}; we omit the details.
\end{proof}

\subsubsection{Comments on higher-order finite-differenced data\label%
{HigherFOU}}

Given the claim that one can always reduce a time series to one with $H<5/8$%
, one ought to check that, in the case of the fOU process, the parameter $%
\theta $ remains explicitly accessible after the finite-difference
transformation. We first leave it to the reader to check that for fOU with
any $H<1$, and any positive real $d$, the fractional finite-differenced
process $D^{d}Z^{\theta }$ has an auto-correlation function which decays
like $k^{2H-2d-2}$, so that the memory length parameter $H^{\prime }=H-d$
satisfies the condition $H^{\prime }<5/8$ as soon as $d>3/8$. Certainly, one
may thus always consider the first-order difference process
\begin{equation*}
X_{k}^{(1)}=X_{k+1}-X_{k}
\end{equation*}%
as suggested in Section \ref{IRC}, and one can write $X^{\left( 1\right)
}=Z^{\left( 1\right) ,\theta }+Y$ and the reader may check that an estimate
such as (\ref{improved lemma 12}) still holds, so that the proposition above
holds for $X^{\left( 1\right) }$ for any $H<1$.

For the quadratic Hermite variation, it remains only to investigate the form
of the variance parameter $\lambda _{\phi _{2}}(Z^{\left( 1\right) ,\theta
})=r_{Z^{\left( 1\right) ,\theta }}\left( 0\right) $ to which $\tilde{Q}%
_{\phi _{2},n}\left( X^{\left( 1\right) }\right) $ converges. The
computation is tedious and presumably known. We find that the variance of $%
Z^{\left( 1\right) ,\theta }$ is
\begin{eqnarray*}
r_{Z^{\left( 1\right) ,\theta }}\left( 0\right) &=&\frac{e^{\theta }}{%
2+e^{-2\theta }}r_{Z^{\theta }}\left( 0\right) \\
&=&\frac{e^{\theta }}{2+e^{-2\theta }}H\Gamma (2H)\left\vert \theta
\right\vert ^{-2H}.
\end{eqnarray*}%
By considering the Hermite quadratic variation instead, as we did in Section %
\ref{FOUEx}, one avoids the singularity at the origin in this function,
obtaining that the function to be inverted to transform $\tilde{Q}%
_{H_{2},n}\left( X^{\left( 1\right) }\right) $ into an estimator of $\theta $
is simply
\begin{equation}
\mu _{H_{2},Z^{\left( 1\right) }}(\theta ):=\frac{e^{\theta }}{2+e^{-2\theta
}}\left( H\Gamma (2H)\left\vert \theta \right\vert ^{-2H}-1\right) .
\label{Muzy1}
\end{equation}%
One is confronted in $\mu _{H_{2},Z^{\left( 1\right) }}$ with a function
which has a single negative global minimum at $\theta _{\min }$, and tends
to $+\infty $ at $0$ and at $+\infty $. Thus in practice, when working with
this first finite difference, additional information is needed to find out
on which side of $\theta _{\min }$ the parameter $\theta $ would be located.
On each of the two intervals $(0,\theta _{\min }]$ and $[\theta _{\min
},+\infty )$ where $\mu _{H_{2},Z^{\left( 1\right) }}$ is a diffeomorphism,
despite the transcendental nature of this function of $\theta $, it is a
simple matter of inverting it numerically to access an asymptotically normal
estimator for $\theta $ for any $H\in \left( 0,1\right) $. It is immediate
that $\mu _{H_{2},Z^{\left( 1\right) }}$ behaves like $\left\vert \theta
\right\vert ^{-2H}$ near $0$, and like an exponential near $+\infty $. Using
Theorem \ref{Berryesseentheta} to prove a Berry-Esséen speed of convergence
for the $\theta $ estimator $\left( \mu _{H_{2},Z^{\left( 1\right) }}\right)
^{-1}\left( \tilde{Q}_{\phi _{2},n}\left( X^{\left( 1\right) }\right)
\right) $, one follows the argument leading to Proposition \ref%
{BerryHermiteThetaProp}, to obtain this $n^{-1/2}$ speed in Wasserstein
distance. The moment condition of Theorem \ref{Berryesseentheta} is
automatically satisfied in the case $\theta <\theta _{\min }$ because the
second derivative of $\left( \mu _{H_{2},Z^{\left( 1\right) }}\right) ^{-1}$
is bounded in that case. In the case $\theta >\theta _{\min }$ the function $%
\left( \left( \mu _{H_{2},Z^{\left( 1\right) }}\right) ^{-1}\right) ^{\prime
\prime }$ has logarithmic growth at $+\infty $; the moment condition in
Theorem \ref{Berryesseentheta} is then also satisfied since chaos variables
have moments of all orders, and certainly thus logarithmic ones. All further
details all omitted. Summarizing, as in Proposition \ref{OptimalBEProp}, we
obtain the following.

\begin{proposition}
Let $\mu _{H_{2},Z^{\left( 1\right) }}$ be as in (\ref{Muzy1}), $\theta
_{\min }$ be that function's global minimizer, and assume either $\theta
<\theta _{\min }$ or $\theta >\theta _{\min }$. Let $\theta ^{\ast }:=\mu
_{H_{2},Z^{\left( 1\right) }}\left( \theta \right) $ and $g:=\left( \mu
_{H_{2},Z^{\left( 1\right) }}\right) ^{-1}$ and
\begin{equation*}
\tilde{Q}_{f_{2},n}\left( X^{\left( 1\right) }\right) :=\frac{1}{n}%
\sum_{i=i_{0}}^{i_{0}+n-1}H_{2}\left( X_{i}-X_{i-1}\right) .
\end{equation*}%
Then for any $H\in \left( 0,1\right) $, we have
\begin{equation*}
\lim \sqrt{n}\left( \tilde{Q}_{H_{2},n}\left( X^{\left( 1\right) }\right)
-\theta ^{\ast }\right) =:u_{H_{2}}(Z^{\left( 1\right) }))<\infty
\end{equation*}%
and there exists an integer $i_{0}>0$ such that $g\left( \tilde{Q}%
_{H_{2},n}\left( X^{\left( 1\right) }\right) \right) $ is a strongly
consistent estimator of $\theta $ and
\begin{equation*}
d_{W}\left( \sqrt{n}\left( g\left( \tilde{Q}_{H_{2},n}\left( X^{\left(
1\right) }\right) \right) -\theta \right) ,\mathcal{N}(0,g^{\prime }\left(
\theta ^{\ast }\right) ^{2}u_{H_{2}}(Z^{\left( 1\right) }))\right) \asymp
\frac{1}{\sqrt{n}}.
\end{equation*}
\end{proposition}

\section{Application to Ornstein-Uhlenbeck processes : multi-parameter
examples\label{Multi}}

In the previous section, we provided a full study of univariate parameter
estimation for a fractional Ornstein-Uhlenbeck process, including all
details of how to apply our general theory. In this final section of our
article, we give two more examples of applications of our methods. For the
sake of conciseness, we focus on the results, providing only a minimal
amount of computations and proofs, since these are all modeled on the
arguments in Section \ref{AOUP}.

\subsection{OU driven by fractional Ornstein-Uhlenbeck process\label%
{OUFOUsection}}

In this section we assume that $X=\left\{ X_{t},t\geq 0\right\} $ is an
Ornstein-Uhlenbeck process driven by a fractional Ornstein-Uhlenbeck process
$V=\left\{ V_{t},t\geq 0\right\} $. This is given by the following linear
stochastic differential equations
\begin{equation}
\left\{ \begin{aligned} X_0=0; \quad dX_t=-\theta X_tdt+dV_t,\quad t\geq0\\
V_0=0;\quad dV_t=-\rho V_tdt+dB^H_t,\quad t\geq0, \end{aligned}\right.
\label{OUFOU}
\end{equation}%
where $B^{H}=\left\{ B_{t}^{H},t\geq 0\right\} $ is a fractional Brownian
motion of Hurst index $H\in (0,1)$, whereas $\theta >0$ and $\rho >0$ are
two unknown parameters such that $\theta \neq \rho $.

Using the notation (\ref{Ztheta}), the explicit solution to this linear
system, noted for instance in \cite{EEV}, implies the following
decomposition of $X_{t}$:
\begin{equation}
X_{t}=\frac{\rho }{\rho -\theta }X_{t}^{\rho }+\frac{\theta }{\theta -\rho }%
X_{t}^{\theta }.  \label{representationX}
\end{equation}%
On the other hand, we can also write the system (\ref{OUFOU}) as follows
\begin{equation}
dX_{t}=-\left( \theta +\rho \right) X_{t}dt-\rho \theta \Sigma
_{t}dt+dB_{t}^{H}.  \label{SIDE}
\end{equation}%
where for $0\leqslant t\leqslant T$
\begin{equation}
\Sigma _{t}=\int_{0}^{t}X_{s}ds=\frac{V_{t}-X_{t}}{\theta }=\frac{%
X_{t}^{\theta }-X_{t}^{\rho }}{\rho -\theta }  \label{expression
Sigma}
\end{equation}%
We also have
\begin{equation}
X_{t}^{\theta }=Z_{t}^{\theta }-e^{-\theta t}Z_{0}^{\theta }
\label{decomposition of X_m}
\end{equation}%
where
\begin{equation}
Z_{t}^{\theta }=\int_{-\infty }^{t}e^{-\theta (t-s)}dB_{s}^{H}.
\label{expression of Z_m}
\end{equation}%
Moreover, the process $\left( Z_{t}^{\theta },Z_{t}^{\theta ^{\prime
}}\right) $ is an ergodic stationary Gaussian process. As consequence
\begin{eqnarray}
X_{t} &=&\frac{\rho }{\rho -\theta }Z_{t}^{\rho }+\frac{\theta }{\theta
-\rho }Z_{t}^{\theta }-\left( \frac{\rho e^{-\rho t}}{\rho -\theta }%
Z_{0}^{\rho }+\frac{\theta e^{-\theta t}}{\theta -\rho }Z_{0}^{\theta
}\right)  \notag \\
:= &&Z_{t}^{\theta ,\rho }-\left( \frac{\rho e^{-\rho t}}{\rho -\theta }%
Z_{0}^{\rho }+\frac{\theta e^{-\theta t}}{\theta -\rho }Z_{0}^{\theta
}\right)  \label{representationX with Z}
\end{eqnarray}%
and
\begin{eqnarray}
\Sigma _{t} &=&\frac{Z_{t}^{\theta }-Z_{t}^{\rho }}{\rho -\theta }-\frac{%
e^{-\theta t}Z_{0}^{\theta }-e^{-\rho t}Z_{0}^{\rho }}{\rho -\theta }  \notag
\\
:= &&\Sigma _{t}^{\theta ,\rho }-\frac{e^{-\theta t}Z_{0}^{\theta }-e^{-\rho
t}Z_{0}^{\rho }}{\rho -\theta }.  \label{expression}
\end{eqnarray}%
Moreover, $Z^{\theta ,\rho }$ and $\Sigma ^{\theta ,\rho }$ are ergodic
stationary Gaussian processes.

Now, assume that the processes $X$ and $\Sigma $ are observed equidistantly
in time with the step size $\Delta _{n}=1$. We will construct estimators for
$(\theta ,\rho )$. By using the ergodicity of $Z^{\theta ,\rho }$ and $%
\Sigma ^{\theta ,\rho }$, Lemma \ref{asymptotic R_{Q,q}} and (\ref%
{consistency for Z+Y}), we conclude that
\begin{equation*}
\left( Q_{f_{q},n}(X),Q_{f_{q},n}(\Sigma )\right) \longrightarrow \left(
\lambda _{f_{q}}(Z^{\theta ,\rho }),\lambda _{f_{q}}(\Sigma ^{\theta ,\rho
})\right)
\end{equation*}%
almost surely as $n\rightarrow \infty $.

Moreover, by the Gaussian property of $Z^{\theta ,\rho }$ and $\Sigma
^{\theta ,\rho }$, and the expressions $\eta _{X}(\theta ,\rho )$ and $\eta
_{\Sigma }(\theta ,\rho )$ for the variances of $Z^{\theta ,\rho }$ and $%
\Sigma ^{\theta ,\rho }$ which are given respectively in (\ref{2nd moment of
OUFOU X}) and (\ref{2nd moment of OUFOU Sigma}) after Lemma \ref{hypotheses
OUFOU} in the Appendix, we can write
\begin{equation*}
\left( \lambda _{f_{q}}(Z^{\theta ,\rho }),\lambda _{f_{q}}(\Sigma ^{\theta
,\rho })\right) =\delta _{f_{q}}\left( \theta ,\rho \right)
\end{equation*}%
where $\delta _{f_{q}}$ is a function which can be expressed via $\eta
_{X}(\theta ,\rho )$ and $\eta _{\Sigma }(\theta ,\rho )$. Hence, in the
case when the function $\delta _{f_{q}}$ is invertible, we obtain the
following estimator for $\theta $
\begin{equation}
(\widehat{\theta }_{f_{q},n},\widehat{\rho }_{f_{q},n}):=\delta _{f_{q}}^{-1}%
\left[ \left( Q_{f_{q},n}(X),Q_{f_{q},n}(\Sigma )\right) \right] .
\label{estimator
of OUFOU}
\end{equation}

\begin{proposition}
Assume $H\in \left( 0,1\right) $ and $\delta _{f_{q}}$ is a homomorphism.
Let $(\widehat{\theta }_{f_{q},n},\widehat{\rho }_{f_{q},n})$ be the
estimator given in (\ref{estimator of OUFOU}). Then, as $n\longrightarrow
\infty $
\begin{equation}
(\widehat{\theta }_{f_{q},n},\widehat{\rho }_{f_{q},n})\longrightarrow
\left( \theta ,\rho \right)  \label{consistency of OUFOU}
\end{equation}%
almost surely.
\end{proposition}

{\noindent \textbf{Examples.}} In the two following examples, the function $%
\delta _{f_{q}}$ is invertible and explicit, based on the expressions for $%
\eta _{X}(\theta ,\rho )$ and $\eta _{\Sigma }(\theta ,\rho )$ given
respectively in (\ref{2nd moment of OUFOU X}) and (\ref{2nd moment of OUFOU
Sigma}) in the Appendix.

\begin{itemize}
\item Suppose that $f_{q}=H_{q}$. Using (\ref{alphaHqcase}), (\ref{2nd
moment of OUFOU X}) and (\ref{2nd moment of OUFOU Sigma}), we have
\begin{equation*}
\delta _{H_{q}}\left( \theta ,\rho \right) =\frac{q!}{(\frac{q}{2})!2^{q/2}}%
\left( \left( \eta _{X}(\theta ,\rho )-1\right) ^{q/2},\left( \eta _{\Sigma
}(\theta ,\rho )-1\right) ^{q/2}\right) .
\end{equation*}

\item Suppose that $f_{q}=\phi _{q}$ with $\phi _{q}(x)=x^{q}$. From (\ref%
{alphaPhicase}), (\ref{2nd moment of OUFOU X}) and (\ref{2nd moment of OUFOU
Sigma}) we obtain
\begin{equation*}
\delta _{\phi _{q}}\left( \theta ,\rho \right) =\frac{q!}{(\frac{q}{2}%
)!2^{q/2}}\left( \left( \eta _{X}(\theta ,\rho )\right) ^{q/2},\left( \eta
_{\Sigma }(\theta ,\rho )\right) ^{q/2}\right) .
\end{equation*}
\end{itemize}

\begin{theorem}
\label{OUFOUThm}Let $H\in \left( 0,\frac{3}{4}\right) $. Define
\begin{equation}
\Gamma _{f_{q}}(\theta ,\rho )=\left(
\begin{matrix}
u_{f_{q}}(Z^{\theta ,\rho }) & u_{f_{q}}(Z^{\theta ,\rho },\Sigma ^{\theta
,\rho }) \\
u_{f_{q}}(Z^{\theta ,\rho },\Sigma ^{\theta ,\rho }) & u_{f_{q}}(\Sigma
^{\theta ,\rho })%
\end{matrix}%
\right)  \label{matrix Gamma(theta,rho)}
\end{equation}%
where
\begin{equation*}
u_{f_{q}}(Z^{\theta ,\rho },\Sigma ^{\theta ,\rho
})=\sum_{k=0}^{q/2}d_{f_{q},2k}^{2}(2k)!\sum_{j\in \mathbb{Z}^{\ast }}\left(
\frac{E\left( Z_{0}^{\theta ,\rho }\Sigma _{j}^{\theta ,\rho }\right) }{%
\sqrt{r_{Z^{\theta ,\rho }}(0)r_{\Sigma ^{\theta ,\rho }}(0)}}\right) ^{2k}.
\end{equation*}%
Then
\begin{equation}
d_{W}\left( \left( U_{f_{q},n}(X),U_{f_{q},n}(\Sigma )\right) ;\mathcal{N}%
\left( 0,\Gamma _{f_{q}}(\theta ,\rho )\right) \right) \leqslant \frac{C}{%
n^{1/4}}  \label{convergence in law vector
Q}
\end{equation}%
Hence, for any $H\in \left( 0,\frac{3}{4}\right) $,
\begin{equation}
\sqrt{n}\left( \widehat{\theta }_{f_{q},n}-\theta ,\widehat{\rho }%
_{f_{q},n}-\rho \right) \overset{\mathrm{\mathcal{L}}}{\longrightarrow }%
\mathcal{\ N}\left( 0,\ J_{\delta _{f_{q}}^{-1}}(\eta _{X}(\theta ,\rho
),\eta _{\Sigma }(\theta ,\rho ))\ \Gamma _{f_{q}}(\theta ,\rho )\ J_{\delta
_{f_{q}}^{-1}}^{T}(\eta _{X}(\theta ,\rho ),\eta _{\Sigma }(\theta ,\rho
))\right)  \label{convergence in law vector}
\end{equation}%
where $J_{\delta _{f_{q}}^{-1}}$ is the Jacobian matrix of $\delta
_{f_{q}}^{-1}$.
\end{theorem}

\begin{proof}
Combining (\ref{representationX with Z}), (\ref{expression}), Lemma \ref%
{hypotheses OUFOU} and Theorem \ref{CLT for Z+Y}, we obtain (\ref%
{convergence in law vector Q}). Applying Taylor's formula we can write
\begin{equation*}
\sqrt{n}\left( \widehat{\theta }_{q,n}-\theta ,\widehat{\rho }_{q,n}-\rho
\right) =J_{\delta _{f_{q}}^{-1}}^{T}\left( \lambda _{f_{q}}(Z^{\theta ,\rho
}),\lambda _{f_{q}}(\Sigma ^{\theta ,\rho })\right) \left(
U_{f_{q},n}(X),U_{f_{q},n}(\Sigma )\right) +d_{n}
\end{equation*}%
where $d_{n}$ converges in distribution to zero, because
\begin{equation*}
\Vert d_{n}\Vert \leqslant C\sqrt{n}\left\Vert \left( Q_{f_{q},n}(X)-\lambda
_{f_{q}}(Z^{\theta ,\rho }),Q_{f_{q},n}(\Sigma )-\lambda _{f_{q}}(\Sigma
^{\theta ,\rho })\right) \right\Vert ^{2}\longrightarrow 0
\end{equation*}%
almost surely as $n\rightarrow \infty $ by using (\ref{convergence in law
vector Q}). Thus the 2-d random vector in the left-hand side of (\ref%
{convergence in law vector}) is the sum of a term converging in law to $0$
and another converging almost surely to $0$; thus it converges in law to $0$%
, establishing (\ref{convergence in law vector}).
\end{proof}

\textbf{Example}: Here we assume that $f_{q}=\phi _{q}$ and $q=2$, and we
can recompute the expression for the function $\delta _{\phi
_{2}}:(0,+\infty )^{2}$ $\mapsto (0,+\infty )^{2}$ as%
\begin{eqnarray*}
\delta _{\phi _{2}}\left( x,y\right) &=&\left( \eta _{X}(x,y),\eta _{\Sigma
}(x,y)\right) \\
&=&H\Gamma (2H)\times \left\{
\begin{array}{l}
\frac{1}{y^{2}-x^{2}}\left( y^{2-2H}-x^{2-2H},x^{-2H}-y^{-2H}\right) \quad %
\mbox{if }\ x\neq y \\
\left( (1-H)x^{-2H},Hx^{-2H-2}\right) \quad \mbox{if }\ y=x.%
\end{array}%
\right.
\end{eqnarray*}%
Since for every $(x,y)\in (0,+\infty )^{2}$ with $x\neq y$ the Jacobian of $%
\delta _{\phi _{2}}$ computes as%
\begin{equation*}
J_{\delta _{\phi _{2}}}\left( x,y\right) =\Gamma (2H+1)%
\begin{pmatrix}
\frac{\left( 1-H\right) x^{1-2H}\left( x^{2}-y^{2}\right) -x\left(
x^{2-2H}-y^{2-2H}\right) }{\left( x^{2}-y^{2}\right) ^{2}} & \frac{\left(
1-H\right) y^{1-2H}\left( y^{2}-x^{2}\right) -y\left(
y^{2-2H}-x^{2-2H}\right) }{\left( x^{2}-y^{2}\right) ^{2}} \\
\frac{Hx^{-2H-1}\left( x^{2}-y^{2}\right) +x\left( x^{-2H}-y^{-2H}\right) }{%
\left( x^{2}-y^{2}\right) ^{2}} & \frac{Hy^{-2H-1}\left( y^{2}-x^{2}\right)
+y\left( y^{-2H}-x^{-2H}\right) }{\left( x^{2}-y^{2}\right) ^{2}}%
\end{pmatrix}%
,
\end{equation*}%
which is non-zero on in $(0,+\infty )^{2}$. So $\delta _{\phi _{2}}$ is a
diffeomorphism in $(0,+\infty )^{2}$ and its inverse $\delta _{\phi
_{2}}^{-1}$ has a Jacobian%
\begin{equation*}
J_{\delta _{\phi _{2}}^{-1}}\left( a,b\right) =\frac{\Gamma (2H+1)}{\det
J_{F_{2}}\left( x,y\right) }%
\begin{pmatrix}
\frac{Hy^{-2H-1}\left( y^{2}-x^{2}\right) +y\left( y^{-2H}-x^{-2H}\right) }{%
\left( x^{2}-y^{2}\right) ^{2}} & -\frac{\left( 1-H\right) y^{1-2H}\left(
y^{2}-x^{2}\right) -y\left( y^{2-2H}-x^{2-2H}\right) }{\left(
x^{2}-y^{2}\right) ^{2}} \\
-\frac{Hx^{-2H-1}\left( x^{2}-y^{2}\right) +x\left( x^{-2H}-y^{-2H}\right) }{%
\left( x^{2}-y^{2}\right) ^{2}} & \frac{\left( 1-H\right) x^{1-2H}\left(
x^{2}-y^{2}\right) -x\left( x^{2-2H}-y^{2-2H}\right) }{\left(
x^{2}-y^{2}\right) ^{2}}%
\end{pmatrix}%
;
\end{equation*}%
where $\left( x,y\right) =\delta _{\phi _{2}}^{-1}\left( a,b\right) $. Thus
the asymptotic covariance matrix in (\ref{convergence in law vector}) is
explicit. Moreover, similarly to the results obtained in Section \ref{AOUP},
we can prove the following, all details being omitted.

\begin{proposition}
Let $\left( \alpha ,\beta \right) \in \mathbf{R}^{2}$. Under the assumptions
and notation of Theorem \ref{OUFOUThm},

\begin{itemize}
\item if $H\in (0,\frac{5}{8})$,%
\begin{equation*}
d_{W}\left( \alpha U_{\phi _{2},n}(X)+\beta U_{\phi _{2},n}(\Sigma );%
\mathcal{N}\left( 0,\left( \alpha ,\beta \right) \Gamma _{\phi _{2}}(\theta
,\rho )\left( \alpha ,\beta \right) ^{Tr}\right) \right) \asymp \frac{1}{%
\sqrt{n}},
\end{equation*}

\item if $H\in (\frac{5}{8},\frac{3}{4})$,%
\begin{equation*}
d_{W}\left( \alpha U_{\phi _{2},n}(X)+\beta U_{\phi _{2},n}(\Sigma );%
\mathcal{N}\left( 0,\left( \alpha ,\beta \right) \Gamma _{\phi _{2}}(\theta
,\rho )\left( \alpha ,\beta \right) ^{Tr}\right) \right) \leqslant \frac{C}{%
n^{4H-3}}.
\end{equation*}
\end{itemize}
\end{proposition}

\subsection{ Fractional Ornstein-Uhlenbeck process of the second kind\label%
{FOUSKsection}}

The last example we consider is the so-called fractional Ornstein-Uhlenbeck
process of the second kind, defined via the stochastic differential equation
\begin{equation}
S_{0}=0,\mbox{ and }\ dS_{t}=-\alpha S_{t}dt+dY_{t}^{(1)},\quad t\geq 0,
\label{FOUSK}
\end{equation}%
where $Y_{t}^{(1)}=\int_{0}^{t}e^{-s}dB_{a_{s}}$ with $a_{s}=He^{\frac{s}{H}%
} $ and $B=\left\{ B_{t},t\geq 0\right\} $ is a fractional Brownian motion
with Hurst parameter $H\in (\frac{1}{2},1)$, and where $\alpha >0$ is the
unknown real parameter which we would like to estimate. The equation (\ref%
{FOUSK}) admits an explicit solution
\begin{equation*}
S_{t}=e^{-\alpha t}\int_{0}^{t}e^{\alpha s}dY_{s}^{(1)}=e^{-\alpha
t}\int_{0}^{t}e^{(\alpha -1)s}dB_{a_{s}}=H^{(1-\alpha )H}e^{-\alpha
t}\int_{a_{0}}^{a_{t}}r^{(\alpha -1)H}dB_{r}.
\end{equation*}%
Hence we can also write
\begin{equation*}
S_{t}=S_{t}^{\alpha }-e^{-\alpha t}S_{0}^{\alpha }
\end{equation*}%
where
\begin{equation*}
S_{t}^{\alpha }=e^{-\alpha t}\int_{-\infty }^{t}e^{(\alpha
-1)s}dB_{a_{s}}=H^{(1-\alpha )H}e^{-\alpha t}\int_{0}^{a_{t}}r^{(\alpha
-1)H}d\tilde{B}_{r},
\end{equation*}%
where the second equality holds by bijective change of variable, where $%
\tilde{B}$ has the same law as $B$. Using a similar argument to that in
Lemma \ref{asymptotic R_{Q,q}} we have for every $p\geq 1$ and for all $n\in
\mathbb{N}$,
\begin{equation}
\left\Vert Q_{f_{q},n}(S)-Q_{f_{q},n}(S^{\alpha })\right\Vert _{L^{p}(\Omega
)}=\mathcal{O}\left( n^{-1}\right) .  \label{hyp on Z+Y FOUSK}
\end{equation}%
As consequence, by using $S^{\alpha }$ ergodic and (\ref{hyp on Z+Y FOUSK})
we conclude that, almost surely as $n\rightarrow \infty $,
\begin{equation*}
Q_{f_{q},n}(S)\longrightarrow \lambda _{f_{q}}(S^{\alpha }).
\end{equation*}%
\newline
Moreover, by the Gaussian property of $U^{\alpha }$ and (\ref{2nd moment of
FOUSK}) we can write
\begin{equation*}
\lambda _{f_{q}}(S^{\alpha }):=\nu _{f_{q}}(\alpha )
\end{equation*}%
where $\nu _{f_{q}}$ is a function. Hence, in the case when the function $%
\nu _{f_{q}}$ is a homeomorphism, we obtain the following strongly
consistent estimator for $\alpha $
\begin{equation}
\widehat{\alpha }_{f_{q},n}:=\nu _{f_{q}}^{-1}\left[ Q_{f_{q},n}(S)\right] .
\label{estimator
of FOUSK}
\end{equation}

\begin{proposition}
Assume $H\in \left( \frac{1}{2},1\right) $ and $\nu _{f_{q}}$ is a
homeomorphism. Let $\widehat{\alpha }_{f_{q},n}$ be the estimator given in (%
\ref{estimator of FOUSK}). Then, almost surely as $n\longrightarrow \infty $
\begin{equation*}
\widehat{\alpha }_{f_{q},n}\longrightarrow \alpha .
\end{equation*}
\end{proposition}

{\noindent \textbf{Examples.}} In the two following examples, the function $%
\nu _{f_{q}}$ is homeomorphic and explicit.

\begin{itemize}
\item Suppose that $f_{q}=H_{q}$. Using (\ref{alphaHqcase}) and (\ref{2nd
moment of FOUSK}), we have
\begin{equation*}
\nu _{H_{q}}(\alpha )=\lambda _{H_{q}}(S^{\alpha })=\frac{q!}{(\frac{q}{2}%
)!2^{q/2}}\left( \frac{(2H-1)H^{2H}}{\alpha }\beta (1-H+\alpha
H,2H-1)-1\right) ^{q/2}.
\end{equation*}

\item Suppose that $f_{q}=\phi _{q}$ with $\phi _{q}(x)=x^{q}$. From (\ref%
{alphaPhicase}) and (\ref{2nd moment of FOUSK}) we obtain $\nu _{\phi
_{q}}(\alpha )=\lambda _{\phi _{q}}(S^{\alpha })=\frac{q!}{(\frac{q}{2}%
)!2^{q/2}}\left[ \frac{(2H-1)H^{2H}}{\alpha }\beta (1-H+\alpha H,2H-1)\right]
^{q/2}.$

\item The reader will check that in both cases above, the function $\alpha
\mapsto v\left( \alpha \right) $ is monotone (decreasing) and convex from $%
\mathbf{R}_{+}$ to $\mathbf{R}_{+}$, and that the moment condition of
Theorem \ref{Berryesseentheta} on $\left( v^{-1}\right) ^{\prime \prime }$
is satisfied.
\end{itemize}

Now, we study the asymptotic distribution of $\widehat{\alpha }_{f_{q},n}$.
By (\ref{inner product of S_alpha}) which is established in Lemma \ref{ppt
FOUSK} in the Appendix, we have for every $H\in (\frac{1}{2},1)$, $%
\sum_{j\in \mathbb{Z}}\left\vert r_{S^{\alpha }}(j)\right\vert ^{2}<\infty $
and $\kappa _{4}(U_{f_{q},n}(S^{\alpha }))=\mathcal{O}(\frac{1}{n})$. Thus,
applying (\ref{d_{TV} for Z+Y with finite serie}) we deduce the following
result.

\begin{proposition}
\label{asymptotic distribution of Q_{q,n}(FOUSK)} Suppose that $H\in (\frac{1%
}{2},1)$ and $\alpha >0$. Then
\begin{equation*}
d_{W}\left( u_{f_{q}}(S^{\alpha })^{-1/2}U_{f_{q},n}(S),N\right) \leqslant
Cn^{-\frac{1}{4}}
\end{equation*}%
In particular,
\begin{equation*}
\sqrt{n}\left( \widehat{\alpha }_{f_{q},n}-\alpha \right) \overset{law}{%
\longrightarrow }\mathcal{N}\left( 0,u_{f_{q}}(S^{\alpha})\left( \left( \nu
_{f_{q}}^{-1}\right) ^{\prime }(\alpha )\right) ^{-2}\right).
\end{equation*}
\end{proposition}

\noindent \textbf{Quadratic case}. In this case we can improve the rate
convergence of $\widehat{\alpha }_{f_{2},n}$. By using Theorem \ref%
{OptimalTheorem}, the estimates $\kappa _{4}(U_{f_{2},n}(S^{\alpha }))=%
\mathcal{O}(\frac{1}{n})$ and $\left\vert E\left( (F_{f_{2},n}(S^{\alpha
}))^{3}\right) \right\vert =\mathcal{O}(\frac{1}{\sqrt{n}})$, and invoking
the properties of $u_{f_{2}}$ described in the examples (bullet points)
above to invoke Theorem \ref{Berryesseentheta}, we get the following.

\begin{proposition}
Let $H\in \left( \frac{1}{2},1\right) $. Then
\begin{equation*}
d_{W}\left( u_{f_{2}}(S^{\alpha })^{-1/2}U_{f_{2},n}(S),N\right) \asymp
\frac{1}{\sqrt{n}},
\end{equation*}%
and
\begin{equation*}
d_{W}\left( \sqrt{n}\left( \widehat{\alpha }_{f_{2},n}-\alpha \right) ,%
\mathcal{N}(0,u_{f_{2}}(S^{\alpha }))\left( \left( \nu _{f_{2}}^{-1}\right)
^{\prime }(\alpha )\right) ^{-2}\right) \leqslant \frac{C}{\sqrt{n}}.
\end{equation*}
\end{proposition}

\section{Appendix\label{Appendix}}

The following result is a well-known direct consequence of the
Borel-Cantelli Lemma (see e.g. \cite{KN}).

\begin{lemma}
\label{Borel-Cantelli} Let $\gamma >0$ and $p_{0}\in \mathbb{N}$. Moreover
let $(Z_{n})_{n\in \mathbb{N}}$ be a sequence of random variables. If for
every $p\geq p_{0}$ there exists a constant $c_{p}>0$ such that for all $%
n\in \mathbb{N}$,
\begin{equation*}
\|Z_{n}\|_{L^p(\Omega)} \leqslant c_{p}\cdot n^{-\gamma },
\end{equation*}%
then for all $\varepsilon >0$ there exists a random variable $\eta
_{\varepsilon }$ such that
\begin{equation*}
|Z_{n}|\leqslant \eta _{\varepsilon }\cdot n^{-\gamma +\varepsilon }\quad %
\mbox{almost surely}
\end{equation*}%
for all $n\in \mathbb{N}$. Moreover, $\mathbb{E}|\eta _{\varepsilon
}|^{p}<\infty $ for all $p\geq 1$.
\end{lemma}

\noindent

\begin{proof}[Proof of Theorem \protect\ref{CLT}]
Since $\frac{U_{f_{q},n}(Z)}{\sqrt{E\left[ U_{f_{q},n}^{2}(Z)\right] }}\in
\mathbb{D}^{1,2}$, by \cite[Proposition 2.4]{NP2013} we have
\begin{equation*}
d_{TV}\left( \frac{U_{f_{q},n}(Z)}{\sqrt{E\left[ U_{f_{q},n}^{2}(Z)\right] }}%
,N\right) \leqslant 2E\left\vert 1-\left\langle D\frac{U_{f_{q},n}(Z)}{\sqrt{%
E\left[ U_{f_{q},n}^{2}(Z)\right] }},-DL^{-1}\frac{U_{f_{q},n}(Z)}{\sqrt{E%
\left[ U_{f_{q},n}^{2}(Z)\right] }}\right\rangle _{\mathcal{H}}\right\vert .
\end{equation*}%
On the other hand, exploiting the fact that
\begin{equation*}
E\left[ E\left[ \left( I_{2k}(g_{2k,n})\right) ^{2}\right] -\langle
DI_{2k}(g_{2k,n}),-DL^{-1}I_{2k}(g_{2k,n})\rangle _{\mathcal{H}}\right] =0
\end{equation*}%
we obtain
\begin{eqnarray*}
&&E\left\vert 1-\langle D\frac{U_{f_{q},n}(Z)}{\sqrt{E\left[
U_{f_{q},n}^{2}(Z)\right] }},-DL^{-1}\frac{U_{f_{q},n}(Z)}{\sqrt{E\left[
U_{f_{q},n}^{2}(Z)\right] }}\rangle _{\mathcal{H}}\right\vert \\
&\leqslant &\frac{1}{E\left[ U_{f_{q},n}^{2}(Z)\right] }\left(
\sum_{k=1}^{q/2}\sqrt{Var\left( (2k)^{-1}\Vert DI_{2k}(g_{2k,n})\Vert _{%
\mathcal{\mathcal{H}}}^{2}\right) }+\sum_{1\leqslant k\neq l\leqslant
q/2}(2l)^{-1}E\left\vert \langle DI_{2k}(g_{2k,n}),DI_{2l}(g_{2l,n})\rangle
_{\mathcal{H}}\right\vert \right) .
\end{eqnarray*}%
Moreover, by \cite[Lemma 3.1]{NPP} we have
\begin{equation*}
Var\left( (2k)^{-1}\Vert DI_{2k}(g_{2k,n})\Vert _{\mathcal{\mathcal{H}}%
}^{2}\right) =(2k)^{-2}\sum_{j=1}^{2k-1}j^{2}j!^{2}\left( _{j}^{2k}\right)
^{4}(4k-2j)!\Vert g_{2k,n}\underset{j}{\tilde{\otimes}}g_{2k,n}\Vert _{{%
\mathcal{\mathcal{H}}}^{\otimes 4k-2j}}^{2},
\end{equation*}%
and for $k<l$
\begin{eqnarray*}
&&E\left[ \left( (2l)^{-1}\langle DI_{2k}(g_{2k,n}),DI_{2l}(f_{2l,n})\rangle
_{\mathcal{H}}\right) ^{2}\right] \leqslant (2k)!\left(
_{2k-1}^{2l-1}\right) ^{2}(2l-2k)!E\left[ \left( I_{2k}(g_{2k,n})\right) ^{2}%
\right] \Vert g_{2k,n}\underset{2l-2k}{{\otimes }}g_{2k,n}\Vert _{{\mathcal{%
\mathcal{H}}}^{\otimes 4k}} \\
&&+2k^{2}\sum_{j=1}^{2k-1}(l-1)!^{2}\left( _{j-1}^{2k-1}\right) ^{2}\left(
_{j-1}^{2l-1}\right) ^{2}(2k+2l-2j)!\left( \Vert g_{2k,n}\underset{2k-j}{{%
\otimes }}g_{2k,n}\Vert _{{\mathcal{\mathcal{H}}}^{\otimes 2j}}^{2}+\Vert
g_{2k,n}\underset{2l-j}{{\otimes }}g_{2k,n}\Vert _{{\mathcal{\mathcal{H}}}%
^{\otimes 2j}}^{2}\right) .
\end{eqnarray*}%
Combining this together with
\begin{eqnarray*}
E\left[ \left( I_{2k}(g_{2k,n})\right) ^{2}\right] &=&\frac{%
(2k)!d_{f_{q},2k}^{2}}{n}\sum_{i,j=0}^{n-1}\left( \frac{r_{Z}(i-j)}{r_{Z}(0)}%
\right) ^{2k} \\
&\leqslant &\frac{(2k)!d_{f_{q},2k}^{2}}{n}\sum_{i,j=0}^{n-1}\left( \frac{%
r_{Z}(i-j)}{r_{Z}(0)}\right) ^{2}=\frac{(2k)!d_{f_{q},2k}^{2}}{2r_{Z}^{2}(0)}%
E(U_{f_{2},n}^{2}(Z)).
\end{eqnarray*}%
and the fact that for every $1\leqslant s\leqslant 2k-1$ with $k\in
\{1,\ldots ,q/2\}$
\begin{eqnarray*}
&&\Vert g_{2k,n}\underset{s}{\otimes }g_{2k,n}\Vert _{{\mathcal{\mathcal{H}}}%
^{\otimes 4k-2s}}^{2} \\
&\leqslant
&d_{f_{q},2k}^{4}(Z)n^{-2}\sum_{k_{1},k_{2},k_{3},k_{4}=1}^{n}\left( \frac{%
r_{Z}(k_{1}-k_{2})}{{r_{Z}(0)}}\right) ^{s}\left( \frac{r_{Z}(k_{3}-k_{4})}{{%
r_{Z}(0)}}\right) ^{s}\left( \frac{r_{Z}(k_{1}-k_{3})}{{r_{Z}(0)}}\right)
^{2k-s}\left( \frac{r_{Z}(k_{2}-k_{4})}{{r_{Z}(0)}}\right) ^{2k-s} \\
&\leqslant
&d_{f_{q},2k}^{4}(Z)n^{-2}r_{Z}(0)^{-4}%
\sum_{k_{1},k_{2},k_{3},k_{4}=1}^{n}r_{Z}(k_{1}-k_{2})r_{Z}(k_{3}-k_{4})r_{Z}(k_{1}-k_{3})r_{Z}(k_{2}-k_{4})
\\
&=&d_{f_{q},2k}^{4}(Z)\kappa _{4}(U_{2,n}(Z))
\end{eqnarray*}%
we deduce that
\begin{equation*}
\sum_{k=1}^{q/2}\sqrt{Var\left( (2k)^{-1}\Vert DI_{2k}(g_{2k,n})\Vert _{%
\mathcal{\mathcal{H}}}^{2}\right) }\leqslant C_{1,q}(Z)\sqrt{\kappa
_{4}(U_{2,n}(Y))}
\end{equation*}%
and
\begin{eqnarray*}
&&\sum_{1\leqslant k\neq l\leqslant q/2}(2l)^{-1}E\left\vert \langle
DI_{2k}(g_{2k,n}),DI_{2l}(f_{2l,n})\rangle _{\mathcal{H}}\right\vert \\
&=&\sum_{1\leqslant k<l\leqslant q/2}(1+\frac{k}{l})(2l)^{-1}E\left\vert
\langle DI_{2k}(g_{2k,n}),DI_{2l}(f_{2l,n})\rangle _{\mathcal{H}}\right\vert
\\
&\leqslant &\sum_{1\leqslant k<l\leqslant q/2}(1+\frac{k}{l})\left( E\left[
\left( (2l)^{-1}\langle DI_{2k}(g_{2k,n}),DI_{2l}(f_{2l,n})\rangle _{%
\mathcal{H}}\right) ^{2}\right] \right) ^{1/2} \\
&\leqslant &C_{2,q}(Z)\sqrt{E\left[ U_{f_{2},n}^{2}(Z)\right] \sqrt{\kappa
_{4}(U_{f_{2},n}(Z))}+\kappa _{4}(U_{f_{2},n}(Z))},
\end{eqnarray*}%
where
\begin{equation}
C_{1,q}(Z)=\sum_{k=1}^{q/2}d_{f_{q},2k}^{2}(Z)(2k)^{-1}\sqrt{%
\sum_{j=1}^{2k-1}j^{2}j!^{2}\left( _{j}^{2k}\right) ^{4}(4k-2j)!},
\label{C1q}
\end{equation}%
and
\begin{eqnarray}
C_{2,q}(Z) &=&\sum_{1\leqslant k<l\leqslant q/2}(1+\frac{k}{l})\left( \max
\left( ((2k)!)^{2}\left( _{2k-1}^{2l-1}\right) ^{2}(2l-2k)!\frac{%
d_{f_{q},2k}^{4}}{2r_{Z}^{2}(0)},\right. \right.  \notag \\
&&\left. \left.
2k^{2}(d_{f_{q},2k}^{4}+d_{f_{q},2l}^{4})\sum_{j=1}^{2k-1}(l-1)!^{2}\left(
_{j-1}^{2k-1}\right) ^{2}\left( _{j-1}^{2l-1}\right) ^{2}(2k+2l-2j)!\right)
\right) ^{1/2}.  \label{C2q}
\end{eqnarray}%
Furthermore,
\begin{eqnarray*}
d_{TV}\left( \frac{U_{f_{q},n}(Z)}{\sqrt{E\left[ U_{f_{q},n}^{2}(Z)\right] }}%
,N\right) &\leqslant &\frac{C_{q}(Z)}{E\left[ U_{f_{q},n}^{2}(Z)\right] }%
\sqrt{E\left[ U_{f_{2},n}^{2}(Z)\right] \sqrt{\kappa _{4}(U_{f_{2},n}(Z))}%
+\kappa _{4}(U_{f_{2},n}(Z))} \\
&\leqslant &\frac{C_{q}(Z)}{E\left[ U_{f_{2},n}^{2}(Z)\right] }\sqrt{E\left[
U_{f_{2},n}^{2}(Z)\right] \sqrt{\kappa _{4}(U_{f_{2},n}(Z))}+\kappa
_{4}(U_{f_{2},n}(Z))} \\
&=&C_{q}(Z)\sqrt{\sqrt{\kappa _{4}(F_{f_{2},n}(Z))}+\kappa
_{4}(F_{f_{2},n}(Z))},
\end{eqnarray*}%
where
\begin{equation}
C_{q}(Z)=2\max \left( C_{1,q}(Z),C_{2,q}(Z)\right) .  \label{Cq}
\end{equation}%
Thus the first estimate of the theorem is obtained. The upper bound of the
second estimate is proved in \cite[Proposition 6.4]{BBNP}.
\end{proof}

\begin{proof}[Proof of Lemma \protect\ref{dWL1}]
By definition of the Wasserstein distance,%
\begin{eqnarray*}
d_{W}\left( Y+Z,N\right) &=&\sup_{h}\left\vert E\left[ h\left( Y+Z\right)
-h\left( Z\right) \right] +E\left[ h\left( Z\right) \right] -E\left[ h(N)%
\right] \right\vert \\
&\leqslant &\sup_{h}\left\vert E\left[ h\left( Y+Z\right) -h\left( Z\right) %
\right] \right\vert +\sup_{h}\left\vert E\left[ h\left( Z\right) \right] -E%
\left[ h(N)\right] \right\vert \\
&\leqslant &E\left[ \left\vert Y\right\vert \right] +d_{W}\left( Z,N\right)
\end{eqnarray*}%
where in the last inequality we used the fact that $h$ is $1$-Lipshitz.
\end{proof}

\begin{proof}[Proof of Lemma \protect\ref{LBdWlemma}]
An inspection of the proof of the main lower bound result in \cite{NP2013}
shows that their lower bound on $d_{TV}\left( F_{n},N\right) $ is in fact a
lower bound on
\begin{equation*}
\frac{1}{2}\max \left\{ \left\vert E\left( \cos F_{n}\right) -E\left( \cos
N\right) \right\vert ;\left\vert E\left( \sin F_{n}\right) -E\left( \sin
N\right) \right\vert \right\} .
\end{equation*}%
Since $\sin $ and $\cos $ are $1$-Lipshitz functions, by definition of $d_{W}
$, this expression is also a lower bound on $\frac{1}{2}d_{W}\left(
F_{n},N\right) $. This proves the lemma.
\end{proof}

\begin{lemma}
\label{inner product for OUFOU}Let $H\in (0,\frac{1}{2})\cup (\frac{1}{2},1]$%
, $m,m^{\prime }>0$ and $-\infty \leqslant a<b\leqslant c<d<\infty $. Then
\begin{equation*}
E\left( \int_{a}^{b}e^{ms}dB^{H}(s)\int_{c}^{d}e^{m^{\prime
}t}dB^{H}(t)\right) =H(2H-1)\int_{a}^{b}dse^{ms}\int_{c}^{d}dte^{m^{\prime
}t}(t-s)^{2H-2}
\end{equation*}
\end{lemma}

\begin{proof}
We use the same argument as in the proof of \cite[Lemma 2.1]{CKM}.
\end{proof}

\begin{lemma}
\label{hypotheses FOU} Let $H\in (0,\frac{1}{2})\cup (\frac{1}{2},1)$, $%
m,m^{\prime }>0$ and let $Z^{\theta }$ be the process defined in (\ref%
{Ztheta}). Then,
\begin{equation*}
r_{Z^{\theta }}(0)=H\Gamma (2H)\theta ^{-2H}
\end{equation*}%
and for large $|t|$
\begin{equation*}
r_{Z^{\theta }}(t)\sim \frac{H(2H-1)}{\theta ^{2}}|t|^{2H-2}.
\end{equation*}
\end{lemma}

\begin{proof}
see \cite[Theorem 2.3]{CKM} or Lemma \ref{hypotheses OUFOU}.
\end{proof}

\begin{lemma}
\label{hypotheses OUFOU} Let $H\in (0,\frac{1}{2})\cup (\frac{1}{2},1)$, $%
m,m^{\prime }>0$ and let $Z^{m}$ be the process defined in (\ref{expression
of Z_m}). Then,
\begin{equation}
E\left[ Z_{0}^{m}Z_{0}^{m^{\prime }}\right] =\frac{H\Gamma (2H)}{m+m^{\prime
}}\left( m^{1-2H}+{(m^{\prime })}^{1-2H}\right)  \label{first moment of Z_0}
\end{equation}%
and for large $|t|$
\begin{equation}
E\left[ Z_{0}^{m}Z_{t}^{m^{\prime }}\right] \sim \frac{H(2H-1)}{mm^{\prime }}%
|t|^{2H-2}.  \label{inner product of Z}
\end{equation}
\end{lemma}

This implies that for $H\in(0,\frac12)\cup(\frac12,1)$
\begin{equation}
\eta_X(\theta,\rho):=E\left[\left(Z_0^{\theta,\rho}\right)^2\right]=\frac{%
H\Gamma(2H)}{\rho^2-\theta^2}[\rho^{2-2H}-\theta^{2-2H}],
\label{2nd moment of OUFOU X}
\end{equation}
\begin{equation}
\eta_{\Sigma}(\theta,\rho):=E\left[\left(\Sigma_0^{\theta,\rho}\right)^2%
\right]=\frac{H\Gamma(2H)}{\rho^2-\theta^2}[\theta^{-2H}-\rho^{-2H}],
\label{2nd moment of OUFOU Sigma}
\end{equation}
and for every $t>0$
\begin{equation}
E\left[\left(Z_t^{\theta,\rho}\Sigma_t^{\theta,\rho}\right) \right] =E\left[%
\left(Z_0^{\theta,\rho}\Sigma_0^{\theta,\rho}\right) \right]=0.
\label{cov(Z,Sigma)}
\end{equation}

\begin{proof}
By using \cite[Proposition A.1]{CKM}, we can write
\begin{eqnarray*}
E\left[ Z_{0}^{m}Z_{0}^{m^{\prime }}\right] &=&mm^{\prime }\int_{-\infty
}^{0}\int_{-\infty }^{0}e^{mu}e^{m^{\prime }v}E\left(
B_{u}^{H}B_{v}^{H}\right) \ dudv \\
&=&\frac{mm^{\prime }}{2}\int_{0}^{\infty }\int_{0}^{\infty
}e^{mu}e^{m^{\prime }v}\left( u^{2H}+v^{2H}-|v-u|^{2H}\right) \ dudv \\
&=&\frac{\Gamma (2H+1)}{2(m+m^{\prime })}\left( m^{1-2H}+{(m^{\prime })}%
^{1-2H}\right) .
\end{eqnarray*}%
Thus the estimate (\ref{first moment of Z_0}) is proved. Now, let $%
0<\varepsilon <1$
\begin{eqnarray*}
E\left( Z_{0}^{m}Z_{t}^{m^{\prime }}\right) &=&e^{-m^{\prime }t}E\left(
\int_{-\infty }^{0}e^{mu}dB_{u}^{H}\int_{-\infty }^{t}e^{m^{\prime
}v}dB_{v}^{H}\right) \\
&=&e^{-m^{\prime }t}E\left( \int_{-\infty }^{0}e^{mu}dB_{u}^{H}\int_{-\infty
}^{\varepsilon t}e^{m^{\prime }v}dB_{v}^{H}\right) +e^{-m^{\prime }t}E\left(
\int_{-\infty }^{0}e^{mu}dB_{u}^{H}\int_{\varepsilon t}^{t}e^{m^{\prime
}v}dB_{v}^{H}\right) \\
:= &&A+B
\end{eqnarray*}%
where, using \cite[Proposition A.1]{CKM} it is easy to see that $|A|=O\left(
e^{-m^{\prime }t}\right) $. On the other hand, by Lemma \ref{inner product
for OUFOU} and integration by parts and linear changes of variables
\begin{eqnarray*}
B &=&H(2H-1)e^{-m^{\prime }t}\int_{-\infty }^{0}du\ e^{mu}\int_{\varepsilon
t}^{t}dv\ e^{m^{\prime }v}(v-u)^{2H-2} \\
&=&\frac{H(2H-1)}{m+m^{\prime }}\left( \int_{t}^{\infty
}e^{-m(z-t)}z^{2H-2}dz+\int_{\varepsilon t}^{t}e^{-m^{\prime
}(t-z)}z^{2H-2}dz+e^{-m^{\prime }t(1-\varepsilon )}\int_{\varepsilon
t}^{\infty }e^{-m(z-\varepsilon t)}z^{2H-2}dz\right) \\
&=&\frac{H(2H-1)}{(m+m^{\prime })}\left( \frac{t^{2H-2}}{m}+\frac{2H-2}{m}%
\int_{t}^{\infty }e^{-m(z-t)}z^{2H-3}dz+\frac{t^{2H-2}}{m^{\prime }}\right.
\\
&&\left. -\frac{(\varepsilon t)^{2H-2}}{m^{\prime }}e^{-m^{\prime
}(1-\varepsilon )t}-\frac{2H-2}{m^{\prime }}\int_{\varepsilon
t}^{t}e^{-m^{\prime }(t-z)}z^{2H-3}dz+e^{-m^{\prime }t(1-\varepsilon
)}\int_{\varepsilon t}^{\infty }e^{-m(z-\varepsilon t)}z^{2H-2}dz\right) \\
&=&\frac{H(2H-1)}{mm^{\prime }}t^{2H-2}+o\left( t^{2H-2}\right) ,
\end{eqnarray*}%
the last inequality coming from the fact that
\begin{equation*}
\int_{t}^{\infty }e^{-m(z-t)}z^{2H-3}dz\leqslant t^{-1}\int_{0}^{\infty
}e^{-my}dy\rightarrow 0,\quad \mbox{ as }t\rightarrow \infty ,
\end{equation*}%
\begin{eqnarray*}
t^{2-2H}\int_{\varepsilon t}^{t}e^{-m^{\prime }(t-z)}z^{2H-3}dz &\leqslant
&\varepsilon ^{2H-3}t^{-1}\int_{\varepsilon t}^{t}e^{-m^{\prime }(t-z)}dz \\
&=&\varepsilon ^{2H-3}t^{-1}\int_{0}^{(1-\varepsilon )t}e^{-m^{\prime
}y}dy\rightarrow 0,\quad \mbox{ as }t\rightarrow \infty ,
\end{eqnarray*}%
and
\begin{equation*}
t^{2-2H}e^{-m^{\prime }t(1-\varepsilon )}\rightarrow 0,\quad \mbox{ as }%
t\rightarrow \infty .
\end{equation*}%
So, we conclude that the estimate (\ref{inner product of Z}) is obtained.
\end{proof}

\begin{lemma}
\label{ppt FOUSK}Let $H\in (\frac{1}{2},1)$. Then,
\begin{equation}
E\left[ \left( S_{0}^{\alpha }\right) ^{2}\right] =\frac{(2H-1)H^{2H}}{%
\alpha }\beta (1-H+\alpha H,2H-1).  \label{2nd moment of FOUSK}
\end{equation}%
and for large $|t|$
\begin{equation}
r_{S^{\alpha }}(t)=E\left[ S_{0}^{\alpha }S_{t}^{\alpha }\right] =O\left(
e^{-\min \{\alpha ,\frac{1-H}{H}\}t}\right) .
\label{inner product of S_alpha}
\end{equation}
\end{lemma}

\begin{proof}
We prove the first point (\ref{2nd moment of FOUSK}). We have
\begin{eqnarray*}
E\left[ \left( S_{0}^{\alpha }\right) ^{2}\right] &=&H(2H-1)H^{2(1-\alpha
)H}\int_{0}^{a_{0}}dyy^{(\alpha -1)H}\int_{0}^{a_{0}}dx\ x^{(\alpha
-1)H}|x-y|^{2H-2} \\
&=&2H(2H-1)H^{2(1-\alpha )H}\int_{0}^{a_{0}}dyy^{(\alpha
-1)H}\int_{0}^{y}dx\ x^{(\alpha -1)H}(y-x)^{2H-2} \\
&=&2H(2H-1)H^{2(1-\alpha )H}\int_{0}^{a_{0}}dyy^{2\alpha H-1}\int_{0}^{1}dz\
z^{(\alpha -1)H}(1-z)^{2H-2} \\
&=&\frac{(2H-1)H^{2H}}{\alpha }\beta (1-H+\alpha H,2H-1).
\end{eqnarray*}%
Thus (\ref{2nd moment of FOUSK}) is obtained. For the point (\ref{inner
product of S_alpha}) see \cite{KS}.
\end{proof}

\end{document}